\documentclass{amsart}

\usepackage{amsthm, amsfonts, amssymb, amsmath, latexsym, enumerate, times, ulem}
\usepackage[colorlinks]{hyperref}
\hypersetup{citecolor=teal, linkcolor=violet, urlcolor=magenta}
\usepackage{tikz}
\usepackage{pgfplots}
\pgfplotsset{compat=1.15}
\usetikzlibrary{arrows}
\usepackage{booktabs}
\usepackage{multirow}
\usepackage{dynkin-diagrams}
\usepackage[latin1]{inputenc}
\usepackage{mathrsfs}
\usepackage{mathtools}
\usepackage{letltxmacro}
\usepackage{multicol}
\LetLtxMacro\orgvdots\vdots
\LetLtxMacro\orgddots\ddots

\makeatletter
\DeclareRobustCommand\vdots{%
	\mathpalette\@vdots{}%
}
\newcommand*{\@vdots}[2]{%
	\sbox0{$#1\cdotp\cdotp\cdotp\m@th$}%
	\sbox2{$#1.\m@th$}%
	\vbox{%
		\dimen@=\wd0 %
		\advance\dimen@ -3\ht2 %
		\kern.5\dimen@
		\dimen@=\wd2 %
		\advance\dimen@ -\ht2 %
		\dimen2=\wd0 %
		\advance\dimen2 -\dimen@
		\vbox to \dimen2{%
			\offinterlineskip
			\copy2 \vfill\copy2 \vfill\copy2 %
		}%
	}%
}
\DeclareRobustCommand\ddots{%
	\mathinner{%
		\mathpalette\@ddots{}%
		\mkern\thinmuskip
	}%
}
\newcommand*{\@ddots}[2]{%
	\sbox0{$#1\cdotp\cdotp\cdotp\m@th$}%
	\sbox2{$#1.\m@th$}%
	\vbox{%
		\dimen@=\wd0 %
		\advance\dimen@ -3\ht2 %
		\kern.5\dimen@
		\dimen@=\wd2 %
		\advance\dimen@ -\ht2 %
		\dimen2=\wd0 %
		\advance\dimen2 -\dimen@
		\vbox to \dimen2{%
			\offinterlineskip
			\hbox{$#1\mathpunct{.}\m@th$}%
			\vfill
			\hbox{$#1\mathpunct{\kern\wd2}\mathpunct{.}\m@th$}%
			\vfill
			\hbox{$#1\mathpunct{\kern\wd2}\mathpunct{\kern\wd2}\mathpunct{.}\m@th$}%
		}%
	}%
}
\makeatother

\usepackage{array}
\usepackage{comment}
\usepackage{paralist}
\usepackage{xcolor}
\usepackage{soul}
\setstcolor{orange}

\newtheorem{theorem}{Theorem}
\newtheorem{lemma}[theorem]{Lemma}

\newtheorem{corollary}[theorem]{Corollary}
\newtheorem{proposition}[theorem]{Proposition}

\newtheorem{question}[theorem]{Question}

\theoremstyle{definition}
\newtheorem{definition}[theorem]{Definition}
\newtheorem{remark}[theorem]{Remark}
\newtheorem*{proposition*}{Proposition}
\newtheorem{example}[theorem]{Example}

\newcommand{\calC}{{\mathcal C}}

\newcommand{\calL}{{\mathcal L}}

\newcommand{\calP}{{\mathcal P}}

\newcommand{\calO}{{\mathcal O}}
\newcommand{\calX}{{\mathcal X}}
\newcommand{\calY}{{\mathcal Y}}
\newcommand{\bbC}{{\mathbb C}}
\newcommand{\bbF}{{\mathbb F}}
\newcommand{\bbP}{{\mathbb P}}
\newcommand{\bbR}{{\mathbb R}}
\newcommand{\bbZ}{{\mathbb Z}}
\newcommand{\bbQ}{{\mathbb Q}}

\newcommand{\MW}{\rm{MW}(\calX)}
\newcommand{\NS}{\rm{NS}(\calX)}
\def\geq{\geqslant}
\def\leq{\leqslant}

\begin{document}
 
\title[Rational elliptic surfaces]{Rational elliptic surfaces with six singular double fibres}

\author[C. Ciliberto] {Ciro Ciliberto \textsuperscript{1}}
\address{\textsuperscript{1} Dipartimento di Matematica, Universit\`a di Roma Tor Vergata, Via O. Raimondo
 00173 Roma, Italia}
\email{cilibert@axp.mat.uniroma2.it}

\author[A. Grassi]{Antonella Grassi \textsuperscript{2}}
\address{\textsuperscript{2} Dipartimento di Matematica, Universit\`a di Bologna}
\address{\textsuperscript{2} Department of Mathematics, University of Pennsylvania}
 \email{antonella.grassi3@unibo.it}
 \email{grassi@upenn.edu}

\author[R. Miranda]{Rick Miranda\textsuperscript{3}}
\address{\textsuperscript{3} Department of Mathematics, Colorado State University, Fort Collins (CO), 80523, USA}
\email{rick.miranda@colostate.edu}

 \author[A. Verra]{Alessandro Verra\textsuperscript{4}}
\address{\textsuperscript{4}Dipartimento di Matematica, Universit\`a Roma Tre}
 \email{verra@mat.uniroma3.it}
 
 \author[A. Zanardini]{Aline Zanardini\textsuperscript{5}}
\address{\textsuperscript{5} Institut de Math\'ematiques, \'Ecole Polytechnique F\'ed\'erale de Lausanne}
 \email{aline.zanardini@epfl.ch}

 
 
\maketitle

\begin{abstract} 
{A rational elliptic surface with section
is a smooth, rational, complex, projective surface  $\mathcal{X}$
that admits a relatively minimal fibration $f: \mathcal{X}\longrightarrow \bbP^1$
such that its general fibre is a smooth irreducible curve of genus one 
and $f$ has a section. In this paper, we classify rational elliptic surfaces with section that have exactly six singular fibres, each counted with multiplicity two. The fibres that appear with multiplicity exactly two
are either of type $II$ or of type $I_2$ of the Kodaira classification. 
We interpret our classification from various viewpoints: 
a pencil of plane cubic curves, 
the Weierstrass equation, 
a double cover of $\bbF_2$ branched over an appropriate trisection 
of the ruling of $\bbF_2$ plus the negative section, 
a double cover of the plane branched along a quartic curve, plus the datum of a point on the plane.
Moreover, either we give explicit normal forms for the plane quartic curve, or we indicate how to find it.}
\end{abstract}

{\hypersetup{hidelinks}
\tableofcontents}

\section*{Introduction} 
A \textit{rational elliptic surface {with section} (RESS)} 
is a smooth, rational, complex, projective surface  $\mathcal{X}$ 
that admits a relatively minimal fibration $f: \mathcal{X}\longrightarrow \bbP^1$ 
such that its general fibre is a smooth irreducible curve of genus one 
and $f$ has a section. 
The study of RESSs is a very classical subject in algebraic geometry, 
which goes back to the 19th century. 
Important and more recent contributions have been made by various authors 
(see, e.g., \cite {bannai2023ramified, elkies6, kuwata, btes, miranda, persson, shiodaMW, vakil}, etc., after the epochal paper by Kodaira \cite {kodaira}).

On a rational elliptic surface with section, 
there are exactly twelve singular fibres of the elliptic fibration, 
if one counts each fibre with the appropriate multiplicity. 
The fibres that appear with multiplicity exactly two 
are either of type $II$ or of type $I_2$ of the Kodaira classification. 
We will call these RESSs \textit{of special type} $(a,b)$ 
if there are $a$ fibres of type $II$ and $b$ of type $I_2$, with $a+b=6$.
In this paper, which is somehow inspired by \cite {vakil}, 
we classify RESSs {of special type} $(a,b)$. 

One can look at a rational elliptic surface with a section from various viewpoints: 
a pencil of (generically smooth) plane cubic curves, 
the Weierstrass equation, 
a double cover of $\bbF_2$ branched along 
an appropriate trisection of the ruling of $\bbF_2$ plus the negative section, 
a double cover of the plane branched along a quartic curve $C$ 
plus the datum of a point $p$ that is not singular for $C$. 
We interpret our classification from all these viewpoints, and in most cases, we give explicit normal forms for the Weierstrass equation 
and the equation of the plane quartic curve in the latter case. 

Let us describe the contents of the paper. 
Section \ref {sec:gen} is devoted to generalities. 
In particular, we introduce the various models of a RESS 
as we indicated above (see \S \ref {ssec:models}) 
and we recall (see \S \ref {NSMW}) some well known facts 
about the Neron-Severi group and the Mordell-Weil group of sections of a RESS,
that are useful to us. 

Section \ref {sec:I_2} is devoted to the classification of RESSs 
that have six fibres of type $I_2$. 
By using the double cover representation of $\bbF_2$, 
we characterise (see Proposition \ref {prop:cub} and Theorem \ref {thm:cub}) 
the pencils of plane cubics that resolve to a RESS of this type. 
In \S \ref {ssec:weir} we give a Weierstrass normal form for these RESSs 
and in \S \ref {sec:double} we give normal forms
for branch curves of the double plane representation. 
It follows from all these representations that
the number of moduli of these RESSs is two (see \S \ref {ssec: modI2}). 

In Section \ref {sec:II}, we deal with the classification of RESSs
with exactly six fibres of type $II$. 
First, we give a Weierstrass normal form for these RESSs (see \S \ref {ssec:wnf}). 
Then, using the structure of the Mordell-Weil group, 
we prove (see Theorem \ref {thm:cr}) that given a RESS 
with exactly six fibres of type $II$,
any two cubic pencils that resolve to such a RESS
are Cremona equivalent. 
In the case under consideration,
it turns out that the modulus of the smooth genus one fibres of the RESS is constant, 
these curves being equianharmonic, i.e., with $J=0$. 
Pencils of (generically smooth) plane cubics with constant modulus 
have been studied in detail by Chisini in \cite {chisini}, 
who, in fact, also characterised the case in question in which the curves of the pencil are equianharmonic. 
Chisini's approach is via the double plane branched along a quartic representation. 
We revisit Chisini's work in this case in \S \ref {ssec:chis} 
and we give a normal form to the equation of the quartic curve $C$ 
and the point $p$ in this case (see Corollary \ref {cor:norm}). 

The subsequent Sections \ref {sec:mix1}, \ref {sec:mix2}, \ref {sec:mix3} 
are devoted to the special RESSs of mixed cases $(a,b)$ with $a+b=6$ and $ab\neq 0$.
As usual, we classify these cases from various viewpoints and give their normal forms.  \medskip

{\bf Aknowledgements:} 
C.C. and A.G are members of GNSAGA of the Istituto Nazionale di Alta Matematica ``F. Severi''.  AG acknowledges funding from the  European Union- NextGenerationEU under the National Recovery and Resilience Plan (PNRR)- Mission 4 Education and research- Component 2 From research to business- Investment 1.1, Prin 2022 ``Geometry of algebraic structures: moduli, invariants, deformations'', DD N. 104, 2/2/2022, proposal code 2022BTA242- CUP J53D23003720006.
\medskip

\section{Generalities} \label{sec:gen}

\subsection{Rational elliptic surfaces} 
We begin by introducing the main character of this paper.
Given a RESS, there may be many sections of the elliptic fibration;
given two sections, the translation on the general fibre by the difference
extends to an automorphism of the surface,
and hence, up to automorphism, any two sections are indistinguishable.
We rigidify the situation by choosing a section and making that part of the given data:

\begin{definition}\label{def:RESS}
A \textit{rational elliptic surface with section} (RESS) is 
{a pair $(\calX,S_0)$,
where $\calX$ is }
a smooth, rational, complex, projective surface 
that admits a relatively minimal fibration $f: \calX\longrightarrow \bbP^1$
such that the general fibre is a smooth irreducible curve of genus one 
and $f$ has a section $\sigma:\bbP^1 \longrightarrow \calX$, 
which we identify with the horizontal curve $S_0=\sigma(\bbP^1)$. 
\end{definition}

In general, given a RESS $f: \calX\longrightarrow  \bbP^1$, 
one can show that $f$ is given by the linear system $|-K_{\calX}|$. 
Hence, $\calX$ can be obtained as a ninefold blow-up of $\bbP^2$ 
at the base points of a pencil of cubic curves 
such that the general curve of the pencil is smooth.
The last blowup produces a $(-1)$-curve $S_0$, which can be taken as the chosen section for the elliptic surface.

Given a RESS $f: \calX\longrightarrow \bbP^1$, 
there are exactly $12$ singular fibres 
(each counted with a suitable multiplicity). 
In this paper, we consider the case in which there are exactly $6$ singular fibres,
each counted with multiplicity two.
We will call these RESSs \textit {of special type}.
This implies that the only Kodaira types that can occur as singular fibres
are a cuspidal curve (type $II$)
and the union of two smooth rational components that meet twice transversely (type $I_2$).
Therefore, the only combinatorial possibilities are $a$ fibres of type $II$ 
and $b=6-a$ fibres of type $I_2$ 
and we will say that such a RESS is \textit{special of type $(a,b)$}. 
A RESS of special type $(6,0)$ [resp. $(0,6)$] is said to be 
\textit{special  of  type $II$} [resp. \textit{special of type $I_2$}]. 
In all other cases, the RESS is said to be of \textit{mixed type}. 

It is a result of Persson and Miranda (\cite{miranda, persson}) that the cases $a=1$ and $a=5$ do not occur. 

\subsection{Models} \label {ssec:models}
{
There are multiple constructions for RESSs,
some of which involve some extra structure,
which we now describe.}


\subsubsection{The Weierstrass model} 
Let us set $\calL\simeq \calO_{\bbP^1}(1)$. 
{The Weierstrass model} can be defined as the closed subscheme 
of $\bbP(\calL^{-2}\oplus \calL^{-3}\oplus \calO_{\bbP^1})$ 
defined by the equation
\begin{equation}\label{weierstrass}
y^2z=x^3+Axz^2+Bz^3,
\end{equation}
where: $A=A(s,t)\in H^0(\bbP^1,\calL^4)$, $B=B(s,t)\in H^0(\bbP^1,\calL^6)$, 
and $(x:y:z)$ are global homogeneous coordinates 
on the $\bbP^2$-bundle $\bbP(\calL^{-2}\oplus \calL^{-3}\oplus \calO_{\bbP^1})$.
In addition, we have that the discriminant 
\begin{equation}\label{def:delta}
D=D(s,t) \coloneqq 4A(s,t)^3+27B(s,t)^2\in H^0(\bbP^1,\calL^{12})
\end{equation}
vanishes exactly at the points whose $f$--fibres are singular.

{
This construction does not depend on any additional data;
it is determined by the pair ($\calX,S_0)$.
Essentially, on the general fibre $F$, the sections of $3S_0$ 
map $F$ to the plane as a cubic curve, with the point ${S_0|}_F$
going to the flex point at $(x:y:z)=(0:1:0)$
(with flex line $z=0$).
The Weierstrass equation of this cubic curve is then obtained 
from a suitable change of coordinates in the plane.
Making the construction global on $\calX$ gives the Weierstrass model.
}

The type of singular fibre at a point $p=(s:t)\in \bbP^1$, 
according to Kodaira's classification \cite{kodaira}, 
can be explicitly recovered from the following data 
(see, e.g., \cite[Table IV.3.1]{btes}):
\begin{itemize}
\item $a=a(p)\coloneqq$ the order of vanishing of $A$ at $p$,
\item $b=b(p)\coloneqq$ the order of vanishing of $B$ at $p$, and
\item $\delta=\delta(p)\coloneqq$ the order of vanishing of $D$ at $p$.
\end{itemize} 
For example, a singular fibre of type $I_2$ 
is characterised by the triple $(a,b,\delta)=(0,0,2)$; 
a singular fibre of type $II$ is characterised by the triple $(a,b,\delta)=(\geq 1,1,2)$.

Often we will use the non--homogeneous version
\begin{equation}\label{Weq}
y^2 = x^3 + A(t)x+B(t)
\end{equation}
of \eqref {weierstrass},  
where $A$ and $B$ are polynomials in the parameter $t$ for the pencil, 
of degrees $4$ and $6$ respectively.
The non--homogeneous form of \eqref {def:delta} is $D(t) = 4A(t)^3+27B(t)^2$;
this is a polynomial of degree $12$,
and has a double root at the points whose fibres are of type $II$ and $I_2$.
Hence in the cases we will consider it must be a perfect square 
of a sextic polynomial with distinct roots.

\subsubsection{The double cover of $\bbF_2$}\label{ssec:dc}
The Weierstrass equation can be interpreted 
as providing a double cover structure $g: Z \longrightarrow \bbF_2$, 
where $\bbF_2$ is the Hirzebruch surface 
$\bbP(\calO_{\bbP^1}\oplus\calO_{\bbP^1}(2))$ 
(see, e.g., \cite[Lecture III, section 2]{btes}).  
The map $g$  is branched over 
the negative section $B$ (of self-intersection $-2$) of $\bbF_2$
and a trisection defined by the right hand side cubic polynomial in $x$ in \eqref {Weq}. 
The RESS $\calX$ is the minimal resolution of the singularities of $Z$.
If $F$ is the fibre class of $\bbF_2$,
then the trisection is in the linear system $|3B+6F|$
and is disjoint from $B$. 
The pencil of curves of genus $1$ is the pull--back of the ruling of $\bbF_2$,
{
and the chosen section $S_0$ is the pre-image of the negative section $B$.
}
From a birational viewpoint, 
the double cover is simply the quotient of $\calX$ by the involution
sending a general point to its inverse in the group law of its fibre,
{where the zero of the group law is given by $S_0$.}

The singular fibres arise from fibres of the ruling that meet the trisection
in fewer than three distinct points.
To obtain a fibre of type $I_2$, the trisection must have an ordinary double point, with (distinct) tangents different from the ruling.
To obtain a fibre of type $II$, the trisection must be smooth, but it should meet the fibre of $\bbF_2$ in one point with multiplicity three: we call this a \textit{flex}.

\subsubsection{The double cover of $\bbP^2$}

{ Another classical construction is provided by }
a Del Pezzo surface of degree $2$, i.e., 
a double cover of the plane $\pi: Y \longrightarrow \bbP^2$ 
branched along a plane quartic ${C}$,
{together with a chosen point $p$ in the plane. }
The elliptic pencil is the pull-back on $Y$ of the pencil of lines through $p$, 
and $\mathcal X$ is the desingularization of the blow--up of $Y$ 
at $\pi^{-1}(p)$. 

{
We will say that the pair $(C,p)$, 
with $C$ a plane quartic and $p$ a point in the plane 
is the \textit{associated quartic} to the RESS 
or that the RESS is \textit{associated} to the pair $(C,p)$. 

If the point $p$ is not on the branch curve $C$,
then the preimage of the blowup of $p$
produces two disjoint sections $S_0$ and $S_1$ on $\calX$;}
{as in \cite{bannai2023ramified}}, {we refer to this as the \textit{split model}.
Hence, this construction yields the additional data of a second section, disjoint from $S_0$.

If the point $p$ is a smooth point of the branch curve $C$,
then in order to resolve the elliptic pencil on $\calX$,
we must blow up $p$ and the point infinitely near to $p$ along $C$.
The first exceptional curve lifts to the chosen section $S_0$;
The second exceptional curve then lifts to a component of a fibre of the elliptic pencil.
Thus, in this case, the construction produces the extra data 
of a choice of a (singular) fibre of the RESS.
We call this case the \textit{ramified model}}, {again borrowing the terminology from \cite{bannai2023ramified}.}



\begin{lemma}\label{lem:PinCi}
 Let $(\calX,S_0)$ be a RESS of special type and let $(C,p)$ be associated to it. 
If $p \in C$ (i.e., a ramified model) then $p$ is not a flex for $C$. 
\end{lemma}

\begin{proof} 
If $p$ is a flex for $C$, 
then the tangent line to $C$ at $p$ 
corresponds to a fibre that is neither of type $II$ nor of type $I_2$.
\end{proof}

\begin{lemma}\label{lem:PinC} 
Let $(\calX,S_0)$ be a RESS of special type $II$ or $I_2$ 
and $(C,p)$ an associated quartic. 
If $p \in C$ (i.e. a ramified model) 
then $\calX$ has to be of special type $I_2$ 
(and $p$ is not a flex for $C$ by Lemma \ref {lem:PinCi}). 
\end{lemma}

\begin{proof}
The tangent line to $C$ at $p$ corresponds to a fibre of type $I_2$. 
The assertion follows.
\end{proof}

{ Table \ref{table:possiblequartics} provides a pictorial description of the types of double fibres arising from a pair $(C,p)$ and a distinguished line $L$ through $p$.} 
The fibres of type $I_2$ arise either from bitangent lines passing through $p$, 
if $ p \notin C$,  
or from lines joining $p$ with a node of $C$ 
and having intersection multiplicity exactly $2$ with $C$ at the node, 
or from the tangent line to $C$ at $p$ if $p\in C$, 
(which is a \textit{simple tangent}, 
i.e., it has exactly intersection multiplicity $2$ with $C$ at $p$). {The fibres of type $II$ arise from simple flex points of $C$
where the flex line passes through $p$.}

\begin{table}[h!]
\centering
\begin{tabular}{>{\centering\arraybackslash} m{0.45\linewidth}| >{\centering\arraybackslash} m{0.35\linewidth}} 
 \toprule
  $(C,p)$ and $L$  & Singular fibre \\ 
  \midrule \\  
  \resizebox{\linewidth}{!}{\begin{tikzpicture}
\draw[black] plot [smooth, tension=2] coordinates { (-1,0) (1,2) (2,0) (4,2) (5,0)} ;
\draw [black] (-3,2) -- (6,2) node[anchor=north]{$L$};
\filldraw[black] (-1.5,2) circle (2pt) node[anchor=north]{$p$};
\end{tikzpicture}}   &  \multirow{7}{*}{$I_2$} \\ \\
\resizebox{0.8\linewidth}{!}{\begin{tikzpicture}
\draw [black] plot [smooth cycle] coordinates {(0,0) (-1,1) (0,2) (1.5,-1) (2,0) (1.5,1) (1,0) (0,-2)(-1,-1)} ;
\draw [black] (-2,0) -- (4,0) node[anchor=south west]{$L$};
\filldraw[black] (3,0) circle (2pt) node[anchor=south]{$p$};
\end{tikzpicture}} &   \\ \cline{1-1} \\
  \resizebox{0.8\linewidth}{!}{\begin{tikzpicture}
\draw[black] plot [smooth, tension=1] coordinates { (-1,0) (1,2) (2,0) (4,2.5) (5,0)} ;
\draw [black] (-1,2) -- (6,2) node[anchor=south]{$L$};
\filldraw[black] (1,2) circle (2pt) node[anchor=south]{$p$};
\end{tikzpicture}} &   \\ \\
  \resizebox{0.8\linewidth}{!}{\begin{tikzpicture}
\draw [black] plot [smooth cycle] coordinates {(0,0) (-1,1.5) (1,2) (2,-1.5) (3,-1.5) (3.5,0) (3,1.5) (2,1.5) (1,-2) (-1,-1.5)};
\draw [black] (-2,0) -- (5,0) node[anchor=south]{$L$};
\filldraw[black] (3.5,0) circle (2pt) node[anchor=south west]{$p$};
\end{tikzpicture}} &   \\ \midrule
\resizebox{0.8\linewidth}{!}{\begin{tikzpicture}
\clip(-3,-2) rectangle (3,1.5);
\draw[line width=1pt,smooth,samples=100,domain=-3:3] plot(\x,{(\x)^(4)+2*(\x)^(3)});
\draw [line width=1pt,domain=-3:3.] plot(\x,{(-0-0*\x)/1}) node[anchor=south east]{$L$};
\filldraw[black] (1.5,0) circle (2pt) node[anchor=south]{$p$};
\end{tikzpicture}}   &  $II$ \\
  \bottomrule                                         
\end{tabular}
\smallskip
\caption{The possible pairs $(C,p)$.}
\label{table:possiblequartics}
\end{table}

\subsubsection{The pencil of cubics} 
{The final construction we will consider is given by }
the plane blown up at nine base points of a pencil of (generically smooth) cubics, 
as we mentioned above. The RESS is obtained by blowing up the base points of the pencil, and we will say that the pencil \textit{resolves} to the RESS.
{In this case the pencil alone does not determine $(\calX,S_0)$;
one must identify the `last' blowup of the resolution of the pencil,
which gives the chosen section $S_0$.}

\subsection{The Mordell-Weil group and the Neron-Severi group of sections}
\label{NSMW}
Let $\MW$ be the set of sections of a RESS $(\calX,S_0)$.
The zero section $S_0$ is a member of $\MW$,
and $\MW$ has a group structure
induced by the group structure on the general fibre;
$S_0$ is the identity of the group. 
$\MW$ is called the \textit{Mordell-Weil group} of $\calX$. 

Let $\NS$ denote the \textit{Neron-Severi group} of $\calX$;
it is a free abelian group of rank $10$,
and the intersection form on $\calX$ gives $\NS$
the structure of a unimodular lattice of signature $(1,9)$,
which is diagonalizable.

Let $U$ be the rank two sublattice of $\NS$
generated by $S_0$ and the fibre class $F$.
This is a unimodular sublattice of signature $(1,1)$ and therefore splits off $\NS$;
the orthogonal complement $U^\perp$
is a rank $8$ sublattice, also unimodular, negative definite, and even;
it is abstractly isomorphic to the $-E_8$ lattice (see \cite [p. 75]{btes}).

Let $R$ be the sublattice of $U^\perp$
generated by fibre components that do not meet $S_0$.
In our case, only type $I_2$ fibres.
So $R$ has a generator for each $I_2$ fibre,
which has two components, each a $(-2)$-curve;
one of which meets $S_0$ (we will systematically call this component $D_i$
where $i\in \{1,\ldots, b\}$ is an index indicating which $I_2$ fibre we are looking at)
and the other of which is a generator for $R$
(we will systematically call this component $C_i$).
Since these generators $\{C_i\}_{1\leq i\leq b}$ live in different fibres,
they are orthogonal in $\NS$,
and so $R$ is a diagonalizable sublattice of $U^\perp$,
freely generated as a lattice by the $C_i$'s,
with diagonal entries all equal to $-2$.

There is a short exact sequence
\begin{equation} \label{UperpMW}
0 \to R \to U^\perp \to \MW \to 0
\end{equation}
(see \cite [p. 70]{btes}) where the map  $\pi:U^\perp \to \MW$
is the summation map:
a divisor on $\calX$
which restricts to a divisor of degree $0$ on the fibres
can be summed, fibre by fibre in the group law of each fibre,
to produce an element $s-0$ in the group law of each fibre,
where $s$ is a point in that fibre.
As that fibre moves, $s$ generates a section.
This is a homomorphism of groups,
but does not respect the intersection form.

There is a function $\sigma:\MW \longrightarrow U^\perp$
which takes a section $S$
and sends it to $\sigma(S) = S-S_0-(k+1)F$
where $k = S\cdot S_0$.
We stress that this function is not a group homomorphism.
Since all sections have self-intersection $-1$,
we see that $\sigma(S)\cdot S_0 = k-(-1)-(k+1) = 0$
and clearly $\sigma(S)\cdot F = 0$ also.
Hence $\sigma(S) \in U^\perp$ as claimed.

This map $\sigma$ is injective:
if $S_1-S_0 - (k_1+1)F$ is linearly equivalent on $\calX$ to $S_2-S_0-(k_2+1)F$,
then $S_1-S_2 + \alpha F$ is linearly equivalent to $0$ for some $\alpha$,
and so, on the general fibre, it restricts to a divisor linearly equivalent to 0.
This forces $S_1 = S_2$.

 In fact, we have 
\[
\pi\circ\sigma = {\rm id_{\MW}}.
\]
To see this, start with a section $S$.
Then $\pi(\sigma(S)) = \pi(S-S_0-(k+1)F) = S$
since the summation map can ignore the fibre contribution
and clearly $S-S_0$ produces the section $S$ with the above summation construction.

A quick computation shows that
\begin{equation}\label{sigmaSsquared}
\sigma(S)^2 = -2 -2S\cdot S_0 \;\;\text{ in }\;\; U^\perp
\end{equation}
and hence the sections that are disjoint from $S_0$
map to $(-2)$-classes in $U^\perp$.
Since $U^\perp$ is abstractly isomorphic to $-E_8$ lattice,
there are exactly $240$ such $(-2)$-classes in $U^\perp$.

Define $\MW^0$ to be the set of sections which are disjoint from $S_0$.
If we contract $S_0$,
we obtain a Del Pezzo surface of degree $1$,
which has finitely many $(-1)$-curves,
and exactly $240$ $(-1)$-classes (see \cite [p. 268]{Bou}).
Since every element of $\MW^0$ will be unaffected by this contraction,
$\MW^0$ must be a finite set.

A second computation shows that if $S_1$ and $S_2$ are disjoint from $S_0$,
and hence in $\MW^0$,
then $\sigma(S_1)\cdot \sigma(S_2) = S_1\cdot S_2 - 1$ in $U^\perp$.
Hence:
\begin{equation}\label{SSdot}
\text{If $S_1$ and $S_2$ are disjoint from $S_0$, then }\;\;
S_1\cdot S_2 = 0 \text{ if and only if } \sigma(S_1)\cdot \sigma(S_2) = -1.
\end{equation}

We next note that given $u \in U^\perp$, we have that
\[
\sigma(\pi(u)) - u \in R
\]
since applying $\pi$ to this gives zero.

Suppose that $u \in U^\perp$ and $u\cdot C_i \in \{0,1\}$ for all $i\in \{1,\ldots, b\}$.
Then we claim the section $\pi(u)$ meets each $C_i$ as $u$ does.

To see this, note that the divisor class $u$ will restrict to a divisor class
on the $I_2$ fibre, and the group law on that fibre is $\bbC^*\times \bbZ/2$,
with the $\bbZ/2$ coordinate recording the two components of the fibre.
Hence if $u\cdot C_i$ is even, the summation map associated to $\pi$ leads to a section which is in the $0$ coordinate of the $\bbZ/2$, and hence meets the $D_i$ component.
However if $u\cdot C_i$ is odd, the summation map gives a section which intersects  $C_i$.  This proves the claim.

We are interested in what follows
in producing sets of sections which are all disjoint,
and all disjoint from $S_0$.
With these observations in mind, we have the following:

\begin{lemma}
Suppose $\{u_\alpha\}$ are sets of elements of $U^\perp$
with the following properties:
\begin{itemize}
\item $u_\alpha^2 = -2$ for all $\alpha$,
\item $u_\alpha\cdot u_\beta = -1$ for all $\alpha \neq \beta$,
\item $u_\alpha \cdot C_i \in \{0,1\}$ for all $\alpha$ and $1\leq i\leq b$.
\end{itemize}
Then $\sigma(\pi(u_\alpha)) = u_\alpha$ and
the corresponding sections $S_\alpha = \pi(u_\alpha)$
are all disjoint, and all disjoint from $S_0$,
and $S_\alpha$ meets each $C_i$ as $u_\alpha$ does.
We have $\sigma(S_\alpha) = u_\alpha$.
\end{lemma}

\begin{proof}
By the claim made above, if we set $S_\alpha = \pi(u_\alpha)$,
then as a divisor, $S_\alpha$ meets each $C_i$ as $u_\alpha$ does.
Therefore $\sigma(S_\alpha) = \sigma(\pi(u_\alpha))$ also does,
by the definition of $\sigma$ (it only subtracts $S_0$ and a multiple of $F$ from $S_\alpha$,
both of which do not meet $C_i$).
Therefore $\sigma(\pi(u_\alpha))-u_\alpha$ is orthogonal to each $C_i$,
and hence lives in $R^\perp$.
However it also lives in $R$; hence it is zero.
This proves the first statement.

The disjointness statements come from the first two hypotheses:
by \eqref{sigmaSsquared}, $u_\alpha^2 = -2$ 
implies that $S_\alpha$ is disjoint from $S_0$.
Then if we have two such, 
the second hypothesis implies that $S_\alpha$ and $S_\beta$ are disjoint 
using \eqref{SSdot}.
\end{proof}

It is a standard fact in the algebra of the $E_8$ lattice
that any collection of eight $(-2)$-classes
with the property that any two intersect in $-1$
are in the same orbit of the Weyl group; there is only one orbit of such.
This follows from the correspondence with
 $(-1)$-curves on the blowup of $8$ points in the plane.
The eight $(-2)$-classes in $E_8$ correspond to $8$ disjoint $(-1)$-classes,
and it is then obvious that any two sets of such disjoint exceptional curves
are Cremona equivalent,
and the Cremona group is the Weyl group of $E_8$.

It is also a fact that the Weyl group acts transitively on subsets of $3$ roots, 
which are pairwise orthogonal.
This is because, transferring attention to the corresponding $(-1)$-classes,
a triple of pairwise orthogonal roots corresponds to a triple of $(-1)$-classes,
each meeting the other once.  
The sum of the three $(-1)$-curves is a genus one curve of self-intersection three; 
the dimension of its linear system is three.
It is well-known that this is Cremona equivalent 
to a linear system of cubics through six points.

\subsubsection{The height pairing} 
On $\MW$ one defines a pairing
$$
\langle , \rangle: \MW\times \MW\longrightarrow \bbQ,
$$
called the \textit{height pairing}, 
in the following way (see \cite[Sect. 8] {shiodaMW}). One sets 
$$
\langle S,S' \rangle:= 1+S\cdot S_0+S'\cdot S_0+S\cdot S'- c(S,S'),
$$
where $c(S,S')\in \bbQ$ is a term that depends on 
the behaviour of $S$ and $S'$ with respect to the reducible fibres of the elliptic fibration. 
It is not important for us to give a general formula for $c(S,S')$. 
In particular we will have
$$
\langle S,S \rangle:= 2+2S\cdot S_0- c(S),
$$
where we set $c(S)=c(S,S)$, and we call $\langle S,S \rangle$ the \textit{height} of $S$. Note that if $c(S)$ and $\langle S,S \rangle=0$, then $S\cdot S_0=-1$ hence $S=S_0$. 

Two facts are for us important. The first is:

\begin{proposition}[\cite{shiodaMW}, (810)] \label{prop:tors} 
The section $S\in \MW$ is torsion if and only if 
$\langle S,S \rangle=0$.
\end{proposition}

The second one is the computation of $c(S)$ 
in our specific case in which all the reducible fibres are of type $I_2$. 
In this case a reducible fibre $F$ is of the form $F=C+D$ 
where $C$, $D$ are irreducible $(-2)$--curves such that $C\cdot D=2$, 
and only one of them, precisely $D$, intersects the $0$--section $S_0$. 
Then the local contribution of $F$ to $c(S)$ is $0$, 
if $S$ intersects $D$, it is $1/2$ if $S$ intersects $C$ (see \cite[l.c.]{shiodaMW}). 
Then $c(S)$ is the sum of all local contributions of the reducible fibres. 
So if there are exactly $b$ reducible fibres of type $I_2$, 
and for exactly $m$ of them the contribution to $c(S)$ is non--zero, 
we have
$$
\langle S,S \rangle:= 2+2\,S\cdot S_0- \frac {m}2.
$$

\begin{lemma}\label{cor:E2TorsionIFF}  
Let $(\calX,S_0)$ be a RESS of special type $(a,b)$ 
and $(C,p)$ an associated quartic {with $p \notin  C$}.  
Let $S_0$ and $S$ be the two sections of $\calX$ 
arising from the blow up of the point $p$ and  
let $S_0$ be the $0$-section. Then:
\begin{enumerate}[(i)]
    \item the height of $S$ is $\langle S,S\rangle = 2- \frac{m}{2}$, 
where $m$ is the number of bitangents to $C$ passing through $p$;
  \item $S$ is a torsion section in $\MW$ if and only if 
$C$ has  $4$ bitangents passing through $p$. 
In this case, $S$ is of order $2$.
\end{enumerate}
\end{lemma}

\begin{proof} 
(i) By the hypothesis, there are $b$ reducible fibres of the RESS, 
and exactly $m\leq b$ of them correspond to bitangent lines passing through $p$, 
whereas $b-m$ correspond to lines joining $p$ with a node of $C$. 
Now it is immediate to see that only the reducible fibres 
coming from the bitangents contribute to $c(S)$, proving (i). 

(ii) That $S$ is torsion follows from Proposition \ref {prop:tors}. 
To prove that $S$ has order $2$, 
note that $S$ restricts to a divisor class on each of the $I_2$ fibre, 
and the group law on such a fibre is $\bbC^*\times \bbZ/2$, 
with the $\bbZ/2$ coordinate recording the two components of the fibre.
Hence $2S$ restricts to a divisor class
on each of the $I_2$ fibre that has $0$ as its $\bbZ/2$ components. 
This means that, for each $I_2$ fibre, 
${2}S$ intersects the same component that $S_0$ intersects. 
This implies that $c(2S)=0$ and the height of $2S$ is also zero 
by Proposition \ref {prop:tors}. 
As we noted above, this implies that $2S=S_0$, i.e., $S$ has order 2, as claimed. 
\end{proof}

\section{Rational elliptic surfaces of special  type $I_2$}\label{sec:I_2}

In this section, we will describe the RESSs of special type $I_2$.

\subsection{The basic example}\label{ssec:basic}
Fix two smooth conics $C_1, C_2$ in $\bbP^2$ 
which are bitangent at two points $p_1$ and $p_2$.
Choose a line $L_1$ which is tangent to $C_2$ at a point $q_1$ and is transverse to $C_1$.
Similarly, choose a line $L_2$, tangent to $C_1$ at $q_2$, transverse to $C_2$.
Let $r$ be the point of intersection between $L_1$ and $L_2$.

Let us consider the pencil generated by $C_1+L_1$ and $C_2 + L_2$.
The nine base points of this pencil are:
\begin{itemize}
\item $p_1$ and $p_2$ and the two points are infinitely near to each other in the direction of the common tangents to the conics at these points;
\item $q_1$ and $q_2$ and the two points are infinitely near to each other in the direction of the two lines;
\item the point $r$.
\end{itemize}

\begin{lemma}
The RESS obtained by resolving these nine base points of the pencil
is special of type $I_2$.
\end{lemma}

\begin{proof}
Two of the six $I_2$ fibres are the conic-line members 
that are given as the generators of the pencil.

For each of the four points $p_i,q_j$, 
there is a member of the pencil which is singular there.
These four members are irreducible nodal cubics;
for each of them, 
the first exceptional curve of the (double) blowup at these four points
meets the proper transform of the nodal cubic (which is a smooth rational curve) 
at two points,
and this proper transform plus that first exceptional curve form the two components
of the $I_2$ fibre corresponding to that singular member of the pencil.

This gives the six singular fibres of type $I_2$. 
\end{proof}

Let us call such a pencil a \textit{bitangent conic-line pencil of cubics}.

\begin{lemma}
The set of bitangent conic-line pencils 
form an irreducible $10$-dimensional family of cubic pencils.
\end{lemma}

\begin{proof}  
The family of bitangent pencil of conics is clearly irreducible 
and depends on six parameters. 
The choice of the two conics in the pencil 
depends on two parameters 
and the two tangent lines on two more parameters. 
The assertion follows.
\end{proof}

\subsection{The double cover of $\bbF_2$ representation}

\begin{proposition}\label{prop:three}
Suppose that $(\calX,S_0)$ is a RESS of special type $I_2$.
In the double cover representation of $\bbF_2$,
we have the double cover map $g:Z \longrightarrow \bbF_2$
which is branched over the $(-2)$-section $B$
and a trisection $T$ in the linear system $|3B+6F|$.
Then $T$  splits into $3$ disjoint sections $T_1, T_2, T_3$,
each having self-intersection $2$; 
they are all members of the linear system $|B+2F|$.
Any two of these meet in two distinct points,
{and they do not all meet at any one point,}
giving six nodes. 
\end{proposition}

\begin{proof} 
In order for $\calX$ to have six $I_2$ fibres, 
the trisection $T$ must have six nodes, 
none of which have the fibres as tangents. 
Since $T$ has arithmetic genus $4$, $T$ must be reducible. 
Then at least one $T_1$ of the irreducible components of $T$ 
must be a section of the ruling of $\bbF_2$. 
Since $T$ does not intersect $B$, then  $T_1\in |B+2F|$ and it is therefore smooth. 
Since $(B+2F)\cdot (2B+4F)=2$, 
on $T_1$ there are only $2$ double points of $T$, 
and the remaining $4$ are on $T-T_1$, that has arithmetic genus $1$. 
So $T-T_1=T_2+T_3$, with $T_2,T_3\in |B+2F|$ and the assertion follows. 
\end{proof}

\subsection{Back to the basic example}

Our next objective is to prove the following: 

\begin{proposition}\label{prop:cub}
Every pencil of plane cubics that resolves to a RESS of special type $I_2$ 
is Cremona equivalent to a bitangent conic-line pencil of cubics.
\end{proposition}

To prove this, we first need to describe $\calX$ starting from the double cover $g:Z \longrightarrow \bbF_2$, and then describe some sections of $\calX$ that will be relevant to us. \smallskip

Let us consider the six double  points of the branch curve $T$ of $g:Z \longrightarrow \bbF_2$, and let us denote them as 
$p_{ija}$ and $p_{ijb}$ for $i,j \in \{1,2,3\}$,
where  $T_i \cap T_j = \{p_{ija},p_{ijb}\}$.

We can obtain the smooth elliptic surface $\calX$ by resolving the singularities of $Z$, that are six points of type $A_1$ located over the double points of the branch curve $T$. To make this resolution, we
 blow up the six double points on $\bbF_2$ and normalize the base change of the double cover  $g:Z \longrightarrow \bbF_2$.
This creates  exceptional curves $\bar{C}_{ijk}$
($1 \leq i,j \leq 3$, $k \in \{a,b\}$)
and the proper transforms of the fibres $\bar{D}_{ijk}$.
They are  $(-1)$-curves in this six-fold blowup of $\bbF_2$.

Let us abuse notation and call the proper transforms of the $T_i$ on this blowup $T_i$ as well.
Then $T_i$ meets four of the $\bar{C}$'s,
namely $\bar{C}_{ijk}$ for $j \neq i$ and $k \in \{a,b\}$.

The normalization of the base change of the double cover is nothing but the double cover branched over $B+T_1+T_2+T_3$
(here the $T_i$'s are the proper transforms), and this is the surface $\calX$. 

Then  the fibre components $\bar{C}_{ijk}$ and $\bar{D}_{ijk}$
lift to the fibre components ${C}_{ijk}$ and ${D}_{ijk}$ for the corresponding fibre of $\calX$ (that are $(-2)$-curves on $\calX$);
they meet at two points, forming the $I_2$ fibre.
The curve
${C}_{ijk}$ does not meet $S_0$ and ${D}_{ijk}$ meets $S_0$.\smallskip

Next we want to describe the sections of $\calX$ that do not meet the zero section $S_0$.

Let $S$ be such a section of $\calX$.
Then $\pi(S)$ is a section of the ruling structure on $\bbF_2$,
which is disjoint from $B$;
hence it is a member of the  linear system $|B+2F|$.
However for such a section to lift to a section of $\calX$,
it must split in the double cover,
or be one of the $T_i$'s.
If it is not one of the $T_i$'s,
this forces $\pi(S)$ to meet $T_1+T_2+T_3$ in a divisor which is double.

There are several ways this can happen. Precisely: 
\begin{itemize}
\item[(a)] $\pi(S)=T_i$ for some $i$.
\item[(b)] $\pi(S)$ is different from all $T_i$ and passes through three of the double points.
\item[(c)] $\pi(S)$ is different from all $T_i$ and passes through two of the double points and is tangent to one of the $T_i$'s at one point.
\item[(d)] $\pi(S)$ is different from all $T_i$ and passes through one of the double points,
and is tangent to two of the $T_i$'s at one point each.
\item[(e)] $\pi(S)$ is different from all $T_i$ and is tangent to each $T_i$,
not passing through any of the double points.
\end{itemize}

In case (a), where $\pi(S) = T_i$,
$S$ is a section of $\calX$ of order two (by the same argument as in the proof of (ii) of Lemma \ref {cor:E2TorsionIFF}), and meets the four components $C_{ijk}$.  
These three are all disjoint (and disjoint from $S_0$).

In case (b), $\pi(S)$ cannot pass through $p_{ija}$ and $p_{ijb}$;
if it would, then it will already meet $T_i$ and $T_j$ twice,
and so cannot meet again.  However, the third double point will be contained in one of these, which is a contradiction. We conclude that
$\pi(S)$ must pass through $p_{12a/b},p_{13a/b},p_{23a/b}$;
there are then $8$ possible choices.
For each choice, $S$ and its conjugate section $\bar{S}$ are disjoint,
giving $16$ sections of $\calX$.
Both $S$ and $\bar{S}$ meet the corresponding three $C_{ija/b}$,
and do not meet the $D_{ija/b}$.

We will denote these sections by
$S_{12a,13a,23a}, S_{12a,13a,23b},\ldots,S_{12b,13b,23b}$
and their conjugates
$\bar{S}_{12a,13a,23a}, \bar{S}_{12a,13a,23b},\ldots,\bar{S}_{12b,13b,23b}$

In case (c), $\pi(S)$ must pass through $p_{ija}$ and $p_{ijb}$ for some pair $i,j$,
and is tangent to the third $T$.
Hence $S$ and $\bar{S}$ meet (over the point of tangency)
and both meet $C_{ija}, C_{ijb}$.
For each pair $i,j$, two such sections are tangent to the third.
Therefore on $\bbF_2$ there are $6$ total sections, lifting to $12$ sections on $\calX$.

We will denote these sections by
$S_{ij\alpha},S_{ij\beta}, \bar{S}_{ij\alpha},\bar{S}_{ij\beta}$ 
(the lifts of the two sections passing through $p_{ija},p_{ijb}$ and tangent to the third).

Case (d) is clearly not possible.

In Case (e), we have that $S$ and $\bar{S}$ meet three times,
and they both meet all $D_{ija/b}$'s.

We can record some of the above information in the following table;
the missing rows correspond to sections which have permutations of the subscripts.
\[
\begin{array}{c|c|c|c|c|c|c}
\text{section} & 12a & 12b & 13a & 13b & 23a & 23b \\
T_1 & C & C & C & C & D & D \\
T_2 & C & C & D & D & C & C \\
T_3 & D & D & C & C & C & C \\
S_{12a,13a,23a} \;\text{or}\; \bar{S}_{12a,13a,23a} & C & D & C & D & C & D \\
S_{12a,13a,23b} \;\text{or}\; \bar{S}_{12a,13a,23b} & C & D & C & D & D & C \\
S_{12\alpha} \;\text{or}\; \bar{S}_{12\alpha} & C & C & D & D & D & D \\
\bar{S}_{12\alpha} & C & C & D & D & D & D \\

\end{array}
\]

Now we can give the:

\begin{proof}[Proof of Proposition \ref {prop:cub}]
In the setup above, we start blowing $\calX$ down to the plane.
To do this, we must identify a series of nine $(-1)$-curves to blow down; we start with $S_0$.

Now consider the section $S_{12\alpha}$.
We note that $\pi(S_{12\alpha})$ meets $\pi(S_{12a,13a,23a})$ at only one point
after we blow up the six double points.
Hence in $\calX$, $S_{12\alpha}$ meets only one of the two lifts of $\pi(S_{12a,13a,23a})$
which are denoted by $S_{12a,13a,23a}$ and $\bar{S}_{12a,13a,23a}$.
Choose the one that does not meet $S_{12\alpha}$; we may re-index so that this is
$S_{12a,13a,23a}$.

Similarly, we see that $\pi(S_{12\alpha})$ meets $\pi(S_{12a,13a,23b})$ at only one point
after we blow up the six double points.
Hence in $\calX$, $S_{12\alpha}$ meets only one of the two lifts of $\pi(S_{12a,13a,23b})$
which are denoted by $S_{12a,13a,23b}$ and $\bar{S}_{12a,13a,23b}$.
Choose the one that does not meet $S_{12\alpha}$; we may re-index so that this is
$S_{12a,13a,23b}$.

At this point all three chosen sections 
$S_{12\alpha}$, $S_{12a,13a,23a}$, and $S_{12a,13a,23b}$
are pairwise disjoint.
Using the above table as a reference, we see that these three sections each meet
$C_{12a}$. In addition, among these three sections,
$C_{12b}$ only meets $S_{12\alpha}$,
$C_{23a}$ only meets $S_{12a,13a,23a}$, and
$C_{23b}$ only meets $S_{12a,13a,23b}$.
Therefore we can blow down
$S_{12\alpha}$ and $C_{12b}$,
$S_{12a,13a,23a}$ and $C_{23a}$,
$S_{12a,13a,23b}$ and $C_{23b}$,
to three infinitely near double base points.

We have now blown $\calX$ down $7$ times, to a surface $\calY$,
and have two more to go.
None of the curves so far blown down have met $C_{13b}$,
and therefore this curve persists in this surface as a $(-2)$-curve.
The surface $\calY$ does not have any curves of self-intersection less than $-2$ on it.
Therefore it is either a $2$-fold blowup of $\bbP^2$
(necessarily infinitely near since we have the $(-2)$-curve on it),
or a $1$-fold blowup of $\bbF_2$ (at a point not on the $(-2)$-section).
These two surfaces are abstractly isomorphic in fact,
and so $\calY$ is determined,
and the only negative curves on it form a chain of three curves
with self-intersections $(-1), (-1), (-2)$.

As noted above, the $(-2)$-curve is $C_{13b}$.
The $(-1)$-curve meeting it cannot be any of $C_{12a}, C_{13a},D_{12a},D_{13a}$,
since those four curves (the components of the $12a$ and $13a$ fibres of $\calX$)
started as disjoint from $C_{13b}$
and none of the curves that were blown down above met $C_{13b}$.

Hence, when we blow down this $(-1)$-curve, and then finally $C_{13b}$,
arriving at $\bbP^2$, all four of these components survive. The reader can easily check that they are the conic-line pairs that generate the pencil,
and the resulting pencil is a member of the bitangent conic-line pencil of cubics.\end{proof}

In fact, we have a better result:

\begin{theorem}\label{thm:cub}
Every cubic pencil that resolves to a RESS of special $I_2$ type
is projectively equivalent to a bitangent conic-line pencil of cubics.
\end{theorem}

\begin {proof}Take such a pencil, and factor the Cremona transformation that takes it to a bitangent conic-line pencil into quadratic Cremona transformations based at three (of the nine) base points.  The reader can now easily check that any such quadratic Cremona transformation sends a bitangent conic-line pencil to another.  Hence, it must have started out as one.

A different proof goes as follows.
Consider the singular elements of the cubic pencil.
No element can consist of three lines; therefore, every singular element is either
a nodal cubic, or a conic plus a line meeting in two distinct points.
If a nodal cubic is a singular element, then since that singular element must resolve to an $I_2$ fibre, the node of the cubic must be a double base point of the pencil (a base point and an infinitely near base point).

Since there are only nine base points total, this means that there are at most four nodal cubics as singular elements.  Hence, there are at least two conic plus line members to the pencil.
One can check that the only way that leads to six $I_2$ fibres is the bitangent conic-line configuration. \end{proof}

 \subsection{The Weierstrass equation  viewpoint}\label{ssec:weir}
We have already observed above that trisection of the $\bbF_2$ that forms the
rest of the branch locus for the double cover splits as three sections in the system
$|B+2F|$.  Hence in the Weierstrass equation \eqref{Weq},
the right-hand side splits as a product of three `parabolas', as
\[
x^3+A(t)x+B(t) = (x-Q_1(t))(x-Q_2(t))(x-Q_3(t))
\]
where the three polynomials $Q_i$ are quadratic.
We must have $Q_3 = -Q_1-Q_2$ so that the cubic in $x$ has no square term.

In that case we can choose any two quadratics $Q_1(t)$ and $Q_2(t)$
(subject to the condition that the three meet in six distinct points in $\bbF_2$).
We have the computation that
\[
A(t) = -(Q_1^1+Q_2^2 + Q_1Q_2), \;\;\;\text{and}\;\;\; B(t) = Q_1Q_2(Q_1+Q_2).
\]
The discriminant is then computed as $D(t) = 4A^3+27B^2$, which in terms of the quadratics is
\[
D = -[(Q_1-Q_2)(Q_1+2Q_2)(2Q_1+Q_2)]^2
\]

This gives all Weierstrass equations for the six $I_2$ fibre case.

\subsection{The double cover of $\bbP^2$ representation} \label{sec:double}

Next, we look at the RESS as a double cover $\pi: Y\longrightarrow \bbP^2$ 
branched along a quartic curve $C$,
{where the elliptic pencil is pulled back from the pencil of lines through a point $p$,}
that may, or may not, belong to $C$ (either ramified or split model). 
We want to identify the pairs $(C,p)$ of special type $I_2$. 

\subsubsection{The possible configurations}

We start with the following.

 \begin{proposition}[{\cite[Theorem 1]{bannai2023ramified}}] \label{lema:num_bit}
 Let $C$ be a plane quartic with at worst nodes as singularities. Then, the number of concurrent bitangent lines of $C$ through a point $p \notin C$ is at most four.
 \end{proposition}

\begin{proposition}\label{prop:quarticI2} Let $C$ be a plane quartic and $p$ a point of $C$.  Let $ \Lambda_p$ be the pencil of lines through $p$. 
Then $(C,p)$ is of special type $I_2$ precisely when it belongs to one of the cases below. 

\begin{enumerate}[{Case} (i)]
  \item ($p\notin C$) [The split model] $C$ has (exactly) $n$ nodes as singularities, there are $m=6-n$ bitangent lines of $C$ in $\Lambda_p$, and any other line in $\Lambda_p$ intersects $C$ with multiplicity at most two. If a line in $\Lambda_p$ goes through a node of $C$, then it is not tangent to $C$ at any other point.
Moreover, the possibilities for $n$ and $m$ that can be realised (and are in fact realised)  are given in Table \ref{tab:possibilities}.  
\begin{table}[h!]
\centering
\begin{tabular}{ |c|c|c| } 
 \hline
  $(n,m)$ &  Normal form   & Descripition of nodal quartic $C$ \\ 
  \hline 
  $(2,4)$   & Proposition \ref{prop:special-2nodes} & Irreducible with $2$ nodes \\ 
  $(3,3)$ & Proposition \ref{prop:3nodes3bitNormalForm} & Irreducible with $3$ nodes\\
  $(4,2)$ & Proposition \ref{prop:4nodes2bitNormalForm} & Two irreducible conics ($4$ nodes)\\ 
  $(6,0)$ & Proposition \ref{prop:specialFourLines} & Four irreducible lines ($6$ nodes)\\
  \hline                                           
\end{tabular}
\smallskip
\caption{$(C,p)$ of special type $I_2$, $p \notin C$.}
\label{tab:possibilities}
\end{table}

 \item ($p\in C$) [The ramified model]  
    The line tangent to $C$ at $p$ is a simple tangent line, $C$ has (exactly) five nodes as singularities, and any other line in $\Lambda_p$ intersects $C$ with multiplicity at most two. In particular, $C$ is the union of an irreducible conic and two lines and $p$ lies on the conic.\\
\end{enumerate}
\end{proposition}

\begin{proof}
Recall that the fibres of type $I_2$ arise either from bitangent lines in $\Lambda_p$ if $ p \notin C$ (Case (i)) or from lines in $\Lambda_p$ through nodes or from the (simple) tangent line to $C$ at $p$ if $p\in C$.   

Case (i). By Proposition \ref{lema:num_bit}, there are at most $4$ concurrent bitangent lines, so $ n \geq 2$. Since the pair 
$(C,p)$ is of special type $I_2$, one has $ n \leq 6$ and the  third  
column of Table \ref{tab:possibilities} follows  (an irreducible plane quartic  has at most $3$ nodes). The condition that any other line in $\Lambda_p$ intersects $C$ with multiplicity at most two guarantees that the principal tangents to $C$ at each of the $n$ nodes do not belong to $\Lambda_p$, and that $Y$ has at worst $A_1$ singularities.

 Case (ii). 
The tangent line at $p$ must be simple,  that is, it is not a bitangent, flex of hyperflex tangent because $(C,p)$ is special of type $I_2$. Then $C$ must have $5$ nodes, that is, $C$ is the union of an irreducible conic and two lines.
 \end{proof}
 
 Next, there are {five} cases to be considered. Either $(C,p)$ is as in Case (i) from Proposition \ref {prop:quarticI2} with $(n,m)\in\{(2,4),(3,3),(4,2), (6,0)\}$, or $(C,p)$ is as in Case (ii) therein. We will consider each of these separately.

\subsubsection{Quartics with two nodes  and four bitangents(Case (i) with $(n,m)=(2,4)$)}

 \begin{proposition}\label{prop:special-2nodes}
  If $C$ has (exactly) two nodes and no other singularity, then we can find coordinates $(x:y:z)$ in $\mathbb{P}^2$ such that $p=(0:0:1)$ and 
  \begin{enumerate}[(i)]
      \item  $C$ has an equation of the form
\begin{equation}\label{eq:wol}
    (ax+by)^2(cx+dy)^2+z^2(2q_2(x,y)+z^2)=0,
\end{equation}
where $q_2(x,y)$ is a general  homogeneous polynomial of degree two and $a,b,c,d$ are constants;
\item the  two nodes have coordinates $(-b:a:0)$ and $(-d:c:0)$;
\item  the four lines $\ell_i:\alpha_ix+\beta_iy=0$, $1\leq i\ \leq 4$, are bitangent to $C$ and concurrent at the point $p=(0:0:1)\notin C$, where 
$$q_2(x,y)^2-(ax+by)^2(cx+dy)^2=\prod_{i=1}^4(\alpha_ix+\beta_iy).$$
  \end{enumerate} 
  \end{proposition}

\begin{proof}
Let $f:\calX \longrightarrow \bbP^1$ be the RESS associated to $(C,p)$, 
and consider the two disjoint sections $S_0,S$ contracted at $p$. 
Take $S_0$ as the zero element on $\MW$. 
Then $S$ is a 2-torsion section (see Lemma \ref{cor:E2TorsionIFF}). 
Such a torsion section induces an involution $\tau$ on $\calX$, 
given by translation by $\tau$ on each fibre, 
and $\tau$ leaves the ramification curve globally invariant. 
So there is an involution on $\bbP^2$, that we also denote by $\tau$, 
that leaves $p$ and $C$ invariant.
Now we imitate the proof in  \cite[Th.~5.1]  {kuwata}. 
Namely, we can choose coordinates in $\bbP^2$ 
such that the involution is given by  $(x:y:z)\mapsto (x:y:-z)$ and 
$p=(0,0,1)\not\in C$. 
Since $C$ is fixed by the involution, 
its equation does not contain $z$ with an odd power, so the equation of $C$ is of the form
$$
rz^4+z^2 p_2(x,y)+p_4(x,y)=0
$$
with $r$ a constant and $p_i(x,y)$ forms of degree $i$, for $i=2,4$. 
Now notice that $r\neq 0$, since otherwise $p$ would be a point of $C$. 
Then we may assume $r=1$. 
By completing a square with respect to $z^2$, we can transform it to the form 
\[
(z^2+q_2(x,y))^2-\prod_{i=1}^4(\alpha_ix+\beta_iy)=0.
\]

Now, the two nodes must be fixed under the involution $\tau$, 
otherwise the line joining them would go through $p$ 
and the singular fibres would be of type $I_4$. 
The nodes must have coordinates of the form $(-b:a:0)$ and $(-d:c:0)$ and
\[
 q_2(x,y)^2-\prod_{i=1}^4(\alpha_ix+\beta_iy)=(ax+by)^2(cx+dy)^2. 
\]
 \end{proof}
 
 \begin{remark}\label{remark_twonodes0} 
Up to projective transformations, we can assume that the two nodes 
as in (ii) of Proposition \ref {prop:special-2nodes} are the points $(1:0:0)$ and $(0:1:0)$, 
i.e., that $a=d=0$ and $b\neq 0\neq c$. 
So \eqref {eq:wol} takes the form
$$
hx^2y^2+z^2(2(\lambda x^2+\mu xy+\nu y^2)+z^2)=0
$$
with $h\neq 0$. 
Note that $\lambda\neq 0\neq \nu$. 
Indeed, intersecting with the line $x=0$, we find the equation $z^2(2\nu y^2+z^2)$, which must have two distinct roots besides the double root $z=0$ corresponding to the double point. 
This shows that $\nu\neq 0$, and similarly $\lambda \neq 0$. 

At this point, we can still act with the diagonal affinities
$$
x\mapsto \alpha x, \quad y\mapsto \beta y, \quad \alpha\neq 0\neq \beta
$$
and impose that $\lambda=\nu=1$, so that equation \eqref {eq:wol}  takes the normal form
\begin{equation}\label{(2,4)}
hx^2y^2+z^2(2(x^2+k xy+y^2)+z^2)=0.
\end{equation}
\end{remark}

\begin{remark}\label{remark_twonodes}
We observe that the binodal quartic in Proposition (\ref{prop:special-2nodes}) 
has geometric genus $1$ 
and therefore can be described as the curve obtained by projecting from a point 
the intersection of two quadrics in $\bbP^3$. 
More precisely, we consider the point $(0:0:0:1)\in \bbP^3$ 
with coordinates $(x:y:z:w)$ as the center of the projection and the quadrics
\[
zw=(ax+by)(cx+dy)
\]
and 
\[
w^2+zw+2q_2(x,y)+z^2-(ax+by)(cx+dy)=0.
\]
\end{remark}

{ We have also a geometric description of the projection:}

Let $E \subset \mathbb P^3$  be a smooth quartic elliptic curve. Consider its dual surface $E^* \subset (\mathbb P^3)^*$ (that is, an elliptic scroll),
and the irreducible curve 
$C \subset E^*$ of its double (nodal, not cuspidal) points. This curve $C$ parametrises the bitangent planes to $E$. The curve $C$ has finitely many 4-secant lines $\ell$. Now each such line $\ell$ corresponds to a pencil $P_\ell$ of planes in $\mathbb P^3$, 
cutting on $E$ a $g^1_4$ with exactly four divisors of type $2x + 2y$. These correspond to the four planes
of the pencil $P_\ell$ which are bitangent to $E$. Finally, let $r_\ell \subset \mathbb P^3$ be the line which is the axis of the pencil of planes $P_\ell$.
Projecting $E$ to $\mathbb P^2$ from a (general) point $q\in r_\ell$, we obtain a plane quartic curve $E_q$, birational to $E$,  which is 2-nodal and has four concurrent bitangent lines, i.e., the images of the four bitangent planes to $C$ in $P_\ell$, under the projection from q.

As usual in genus one, the choice of $\mathcal O_E(1)$ does not matter, via the action of ${\rm Aut}(E)$. So, as we already saw, there are two moduli: one
is for $E$ and one for a point moving on the axis $r_\ell$. Also, the moduli space is a ruled surface living in the Hurwitz scheme $\mathcal H_{1,4}$.

\subsubsection{Quartics with three nodes  and three bitangents (Case (i) with $(n,m)=(3,3)$)}

   \begin{lemma}[{\cite{hilton}, Chapter XVII, Sect. 3}]\label{lema:3nodes}
      An irreducible plane quartic that has three nodes has exactly four bitangent lines.
     Its equation can be chosen to be 
     \begin{equation}\label{eq:normal_3nodes}
     (u^2-bcx^2-cay^2-abz^2)^2=(u-u_1)(u-u_2)(u-u_3)(u-u_4),
     \end{equation}
     where $u=fx+gy+hz$ and 
     \begin{align*}
         u_1 &\coloneqq \sqrt{bc}x+\sqrt{ca}y+\sqrt{ab}z, \\
         u_2 & \coloneqq \sqrt{bc}x-\sqrt{ca}y-\sqrt{ab}z,\\
         u_3 & \coloneqq -\sqrt{bc}x+\sqrt{ca}y-\sqrt{ab}z, \text{and}\\
         u_4 & \coloneqq -\sqrt{bc}x-\sqrt{ca}y+\sqrt{ab}z.
     \end{align*}
   The bitangents to $C$ are the four lines given by $u-u_i=0$ (with $1\leq i\leq 4$). 
     \end{lemma}

     \begin{proof} The proof can be found in \cite[l.c.]{hilton}. 
However, we sketch it for the reader's convenience. 
We first observe that we can choose coordinates in $\bbP^2$  
such that the three nodes are the three coordinate points. 
Thus, $C$ has an equation of the form
     \[
     ay^2z^2+bz^2x^2+cx^2y^2+2xyz(fx+gy+hz)=0.
     \]
    The claim follows from the substitutions in the statement.
\end{proof}

\begin{proposition}\label{prop:3nodes3bitNormalForm}
    If $(C,p)$ is special of type $I_2$ and $C$ is irreducible with three nodes, 
then there exist coordinates in $\bbP^2$ such that $C$ is given by an equation of the form:
    \begin{equation}\label{eq:wer}
         ay^2z^2+bz^2x^2+cx^2y^2+2xyz(fx+gy+hz)=0.
    \end{equation}
with   $(a,b,c,f,g,h)$  such that ${\rm rank} (A)=2$, where
  \begin{equation}\label{eq:rew}
  A=\begin{pmatrix}
     f-\sqrt{bc} & g-\sqrt{ca} & h-\sqrt{ab}\\
      f-\sqrt{bc} & g+\sqrt{ca} & h+\sqrt{ab}\\
 f+\sqrt{bc} & g-\sqrt{ca} & h+\sqrt{ab}    
    \end{pmatrix}.
        \end{equation}
        Then $p$ has coordinates $(x_0,y_0, z_0)$ given by the kernel of $A$. 
\end{proposition}

\begin{proof}
This follows from the proof of Lemma \ref{lema:3nodes}. 
With the same notation, we see that equation \eqref{eq:normal_3nodes} shows that 
we can construct a trinodal quartic $C$ with three bitangent lines concurrent at a point $p$ 
if and only if we can find $(a,b,c,f,g,h)\in \mathbb C^6$  as above.  
The point $p$ is such that its homogeneous coordinates are given by $\ker(A)$, 
and the pair $(C,p)$ will be special provided the lines $u-u_i$ (with $1\leq i\leq 4$) 
do not go through the three nodes.
\end{proof}
     
In particular, we can produce concrete examples of special pairs $(C,p)$ 
by choosing a tuple $(a,b,c,f,g,h)$ with $a=b=c=1$, $f+g=h+1$ and  $f,g,h\neq 1$.

\begin{example}
Let $q(x,y,z)=-x^2-y^2-z^2+(2x+3y+4z)^2$, $l_1(x,y,z)=x+2y+3z$, $l_2(x,y,z)=3x+2y+5z$,  $l_3(x,y,z)=x+4y+5z$ and $l_4(x,y,z)=3x+4y+3z$ 
and consider the plane quartic given by $q^2=l_1l_2l_3l_4$. 
This quartic has three nodes at the points $(0:0:1),(0:1:0)$ and $(0:0:1)$. 
Moreover, the three bitangent lines $l_i=0$, with $1\leq i\leq 3$, 
are concurrent at the point with coordinates $(-1:-1:1)$, 
and they do not go through any of the three nodes.
\end{example}

\begin{remark}\label{rem:abc} 
One can reduce the parameters on which the equation \eqref {eq:wer} depends in the following way. 
First of all, \eqref {eq:wer} depends on six homogeneous parameters $(a,b,c,f,g,h)$, so it actually depends on five parameters by making $a=1$. 
Then we can impose the kernel of the matrix $A$ in \eqref {eq:rew}, {that is, the point $p$, } to be $(1:1:1)$. 
This imposes three linearly independent relations between the parameters $(1,b,c,f,g,h)$, reducing them to two.  
\end{remark}

\subsubsection{Quartics with four nodes and two bitangents (Case (i) with $(n,m)=(4,2)$)}

\begin{proposition}[Two irreducible conics]\label{prop:4nodes2bitNormalForm}
If $(C,p)$ is special of type $I_2$ and $C$ has four nodes, then we can find coordinates
$(x : y : z)$ in $\bbP^2$ such that $p=(0:0:1)$ and $C$ has an equation of the form
\begin{equation}\label{eq:xy} 
(xy+az^2)(xy+b(x+y-z)^2)=0
\end{equation}
with $a\neq 0\neq  b$ such that the two conics $xy+az^2=0$ and $xy+b(x+y-z)^2=0$ 
intersect in four distinct points. 
\end{proposition}

\begin{proof} 
If $C$ has four nodes, then it consists of two conics $C_1, C_2$ that intersect in four distinct points. 
Moreover { if  it is of special type $I_2$ } through the point $p$ there must be two lines $r,s$ 
both tangent to the $C_1, C_2$ in distinct points. 
Up to projective transformations we can assume that $p=(0:0:1)$, 
that $r,s$ have equations $x=0$, $y=0$ 
and that the contact points of $C_1$ [resp. of $C_2$] with $r$ and $s$ 
are the points with coordinates $(0:1:0)$ and $(1:0:0)$ [resp. $(0:1:1)$ and $(1:0:1)$]. 
Then $C_1$ [resp. $C_2$] belongs to the pencil of conics 
with equation $\alpha xy+\beta z^2=0$ [resp. $\gamma xy+\delta (x+y-z)^2=0$]. 
Since we do not want $C_1$ or $C_2$ to be a reducible conic, 
we can assume $\alpha= \gamma=1$ and $\beta\neq 0\neq \delta$. 
Eventually, we see that the equation of $C=C_1+C_2$ is of the form \eqref {eq:xy}. 
\end{proof}

\subsection{Quartics with five nodes  (Case (ii))}

\begin{proposition}[An irreducible conic and two lines]\label{prop:special-reducible}
    If $(C,p)$ is special of type $I_2$ and $C$ has five nodes, 
then we can find coordinates $(x:y:z)$ in $\mathbb{P}^2$ such that $C$ has an equation of the form
      \begin{equation}\label{eq:special-reducible}
   xy(axy -(xz+yz-z^2))=0,  
\end{equation}
where $a \neq 0$ and $p=(x_0:y_0:z_0)$ satisfies both 
$x_0y_0\neq 0$ and $$ax_0y_0-(x_0z_0+y_0z_0-z_0^2)=0.$$ 
Note that $C$ has five nodes at the points
\[
(0:1:0),(0:1:1),(1:0:0),(1:0:1) \quad \text{and} \quad (0:0:1).
\]
  \end{proposition}
\begin{proof}
If $C$ has five nodes, 
then $C$ must be the union of an irreducible conic and two general lines. We can choose  coordinates such that
$(0:1:0)$, $(0:1:1)$, $(1:0:0)$ and $(1:0:1)$ are the intersections of the conic with the lines, and the lines have the equations $x=0$ and $y=0$.
We see that the conic must have an equation of the form $axy-xz-yz+z^2=0$ with $a\neq 0$. Moreover, the point $p$ must lie on the conic. 
  \end{proof}

\subsubsection{Quartics with six nodes  (Case (i) with $(6,0$.)}

The following is trivial:

\begin{proposition}\label{prop:specialFourLines}
If the plane quartic $C$ has six nodes, 
we can find coordinates $(x:y:z)$ in $\mathbb{P}^2$ 
such that $C$ has an equation of the form
$$
   xyz(x+y+z)=0
$$
and $p$ is any sufficiently general point of $\bbP^2$. 
\end{proposition}

\subsection{Moduli}\label{ssec: modI2}
It is now straightforward to check that the family of RESSs of special $I_2$ type 
is irreducible and depends on two moduli. 

First, referring to the basic example in Section \ref {ssec:basic}, to  Proposition \ref {prop:cub} and { to Theorem \ref{thm:cub} }
 Proposition \ref {prop:cub}, we have the number of ways of choosing the bitangent conic-line pairs 
up to projective equivalence (which is an $8$-dimensional group).
First, we choose the first conic; there are no moduli for this,
but the projective group is now down to a $3$-dimensional group, the group that preserves this conic.
Second, we choose two points on this conic,
where we will have the bitangencies with the second conic.  
This again has no moduli, 
but the group is now down to a $1$-dimensional group.
Now we choose the second conic which is bitangent at these two points; 
this is a $1$-dimensional family, 
and all are in the same orbit of the $1$-dimensional group that is left; 
hence we have no further continuous group elements to use.
Finally, we have the two tangent lines, each requiring a choice of a point on each conic; this yields two moduli.

The second computation, referring to Proposition \ref {prop:three}, 
is the choice of the three sections of $\bbF_2$.
Each section comes from the $3$-dimensional linear system $|B+2F|$, and so this choice has nine parameters.
The group of automorphisms of $\bbF_2$ is $7$-dimensional;
the difference gives the two moduli.

Third, referring to the Weierstrass equation analysis in Section \ref {ssec:weir}, 
we have seen that the surface is determined by two quadratic polynomials, 
which have $6$ parameters.
Three of them are taken up by the automorphisms of the $\bbP^1$ base of the pencil;
the last is the scaling parameter of the Weierstrass equation.
This gives two moduli again.

Finally, referring to the double cover of $\bbP^2$ representation in Section \ref {sec:double}, we have provided, for all possible cases listed in Proposition \ref {prop:quarticI2}, normal forms for the pair $(C,p)$ and each of them depends on two parameters. \medskip

\begin{remark}
We noted above that the double-plane constructions provide additional structure to the RESS. In the split case, the construction yields two sections of the rational elliptic surface by considering the two pre-images of the blowup of the point $p$.
 If we choose one of the sections to be the zero-section $S_0$, and we call $S_1$ the other section, then Lemma \ref{cor:E2TorsionIFF}  proves that if the RESS is of special type $(a,b)$, $S_1$ is torsion, of order two, if and only if there are $4$ bitangent lines through the point $p$.

In the special $I_2$ case, in each of the six singular fibres, it can happen that $S_0$ and $S_1$ either meet the same component of the singular fibre, or different components.
Hence, we obtain a combinatorial invariant of the construction,
namely the number $k$ of singular fibres in which $S_0$ and $S_1$ meet the same component.

We recall that singular $I_2$ fibres arise from lines through $p$ that either pass through a node of the quartic $C$ or are bitangent to $C$. In the case where the line passes through a node of $C$, we see that $S_0$ and $S_1$ meet the same component of the singular fibre (namely the component lying above the line). In the case where the line is bitangent to $C$, the two sections meet different components (the line splits into the two components of the $I_2$ fibre in the double cover). Therefore, this invariant $k$ measures the number of nodes of the quartic curve $C$, in the split case; we called this number $n$ above, so that $k=n$.

One may then consider how the construction behaves under birational maps.  The invariance of the number $k$ shows that $n$ is also a birational invariant, and no Cremona transformation of the quartic curve can change the number of nodes for this construction, if the base point $p$ for the pencil of lines remains off the quartic.

 There is only one family when the base point lies on the quartic, and it is not difficult to see that this family (case (ii) of Proposition \ref{prop:quarticI2}) can be brought to case (i) with $(n,m)=(6,0)$  via a Cremona transformation.

Note that moreover the limit case for $a\longrightarrow 0$ in Proposition \ref {prop:special-reducible}  (that still refers to case (ii) of Proposition \ref{prop:quarticI2}) is the equation appearing in Proposition \ref {prop:specialFourLines} (that refers to case (i) with $(n,m)=(6,0)$ of Proposition \ref{prop:quarticI2}). The point $p \in C$   in the limit becomes a point outside the four lines. 

Hence, we conclude that there are four families (up to birational equivalence) for the double plane construction (with the extra data it entails, as noted above); each has two moduli.
\end{remark}


These considerations lead to the following observation: in the special $I_2$ case,
the six points on $\bbP^1$ over which there are the six $I_2$ fibres of $\calX$
are not general, but form a family (in the $\bbP^6$ of the space of six points)
of dimension five, a hypersurface.
In Section \ref {ssec:weir}, we have parametrised that hypersurface by two quadratic polynomials $Q_1$ and $Q_2$.

Some interesting questions are in order:

\begin{question}\label{quest:mm}
What is the equation of this hypersurface in $\bbP^6$?
What are the conditions on the corresponding curves of genus two 
(branched over these six points) that give a divisor in $M_2$?
\end{question}

\section{Rational elliptic surfaces with section of special  type $II$}\label{sec:II}

In this section, we will describe RESSs of special type $II$.

\subsection{The Weierstrass Equation}\label{ssec:wnf}
In this case since all singular fibres have $J=0$,
we must have $A=0$ in the Weierstrass equation.
Hence, it is of the form
\[
y^2 = x^3 + B(t)
\]
for a sextic polynomial $B$ with six distinct roots,
which define the positions of the six singular fibres of type $II$. 
Contrary to what happens in the special $I_2$ case, 
here the six points on $\bbP^1$ that correspond to the six singular fibres of type $II$ 
are general. 
Thus, the number of moduli is three.

\subsection{The $E_8$ lattice}
Since there are no components of singular fibres not meeting the zero section,
the sublattice $R$ of Section \ref{NSMW} is trivial,
and so by \eqref{UperpMW} we see that the $E_8$ lattice $U^\perp$
is isomorphic to the Mordell-Weil group of section $\MW$.

Every blowdown of the RESS $\calX$ to $\bbP^2$
must contract nine disjoint sections,
all of which are $(-1)$-curves.
There are infinitely many sections, but only finitely many collections of $8$ disjoint sections, all disjoint from $S_0$.

This is explained by considering the map $\sigma$ introduced in Section \ref{NSMW}.
All sections disjoint from $S_0$ map to $(-2)$-classes in $U^\perp \cong E_8$,
and two such sections are disjoint from each other
if and only if their intersection in $E_8$ is $-1$ (recall \eqref {SSdot}).

Since this Weyl group corresponds to the Cremona group of
the birational transformations of the plane based at the eight points, we have the following:

\begin{theorem}\label{thm:cr}
Given a RESS $(\calX,S_0)$  of special type $II$,
any two cubic pencils which resolve to $\calX$ are Cremona equivalent.
\end{theorem}

Indeed, there is a more general fact, 
since the above argument only required that all fibres of $\calX$ are irreducible:

\begin{theorem}
Given a RESS $(\calX,S_0)$ with all irreducible fibres,
any two cubic pencils which resolve to $\calX$ are Cremona equivalent.
\end{theorem}

A standard way of exhibiting eight elements of $E_8$ pairwise intersecting in $-1$
is to use the negative definite Euclidean form on $\bbR^8$,
and consider the $8$ vectors in the rows of the following table:
\[
\begin{array}{cccccccc}
1 & 1 & 0 & 0 & 0 & 0 & 0 & 0 \\
1 & 0 & 1 & 0 & 0 & 0 & 0 & 0 \\
1 & 0 & 0 & 1 & 0 & 0 & 0 & 0 \\
1 & 0 & 0 & 0 & 1 & 0 & 0 & 0 \\
1 & 0 & 0 & 0 & 0 & 1 & 0 & 0 \\
1 & 0 & 0 & 0 & 0 & 0 & 1 & 0 \\
1 & 0 & 0 & 0 & 0 & 0 & 0 & 1 \\
1/2 & 1/2 & 1/2 & 1/2 & 1/2 & 1/2 & 1/2 & 1/2
\end{array}
\]

Alternatively, we can use the standard $E_8$ graph to describe the phenomenon.
This is the graph

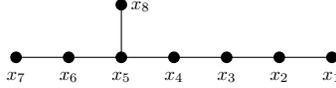
\[
 \dynkin[labels={x_7,x_8,x_6,x_5,x_4,x_3,x_2,x_1}, scale=2,root radius =1pt,o/.style={fill=black}]E{8}
\]

Here, the vertices represent $(-2)$-classes in $E_8$, and the adjacencies mean that those classes meet in $+1$.
The following are then $8$ classes, all of which have self-intersection $-2$, and meet each and every one of the others in $-1$:
\[
\begin{array}{ccccccccc}
\text{class in }E_8 &x_1&x_2&x_3&x_4&x_5&x_6&x_7 &x_8\\
r_1: &1 &0 &0 &0 &0 &0 &0 &0 \\
r_2: & 1& 1&0 &0 &0 &0 &0 &0 \\
r_3: & 1&1 &1 & 0&0 &0 &0 &0 \\
r_4: & 1&1 &1 &1 &0 &0 &0 &0 \\
r_5: & 1&1 &1 &1 &1 &0 &0 &0 \\
r_6: & 1& 1& 1& 1& 1& 1& 0& 0\\
r_7: & 1&1 &1 &1 &1 &1 &1 &0 \\
r_8: & 2&3 &4 &5 &6 &4 &2 &3 
\end{array}
\]

\begin{remark}Recall that if a pair $(C,p)$
defines a rational elliptic surface of special type $II$, then 
$p \notin C$ (Lemma \ref{lem:PinC}). 
The corresponding model is necessarily split.
\end{remark}

\subsection{The double cover of $\bbP^2$ approach} \label{ssec:chis}
A smooth curve of genus one with $J=0$ is said to be \textit{equianharmonic}. If $f: \calX\longrightarrow \bbP^1$ is a RESS of special type $II$, then all smooth fibres of $f$ are equianharmonic because clearly the moduli map for the family of fibres of $f$ is constant. These RESSs correspond to pencils of plane curves whose smooth members are all equianharmonic. So, for us, it is relevant to classify such RESSs. This question has been treated as a particular case in \cite {chisini} and different aspects are also explored in \cite {Ye1, Ye2}.

Taking advantage of the double cover of $\bbP^2$ representation of a RESS, Chisini observes in \cite{chisini} that there exists a one-to-one correspondence between 
pencils of plane cubics with constant $J$--invariant 
and pencils of lines in the plane through a point $p$ 
that cut out on a plane quartic $C$ 
quadruples of points all having the same cross-ratio. 

When $J= 0$, Chisini provides an explicit equation for the plane quartic and describes many of its properties. 
This section revisits Chisini's work (for $J\equiv 0$) 
and summarises his results. 
We begin by introducing the following two definitions.

\begin{definition}\label{lambda-special}
A plane quartic $C\subset \bbP^2$ is $\lambda$-special 
if there exists a pencil of lines (in the same $\bbP^2$) 
that cuts the quartic in quadruples of points 
that have the same cross-ratio $\lambda$. 
Moreover, if $1-\lambda+\lambda^2=0$, 
we say that $C$ is \textit{equianharmonic}.
\end{definition}

\begin{definition}\label{lambda-special2}
A nonzero binary form of degree four is said to be \textit{equianharmonic} 
if it defines a quadruple of points with cross-ratio $\lambda$,
where $1-\lambda+\lambda^2=0$.
\end{definition}

As shown by Chisini, 
any $\lambda$-special plane quartic $C$ 
corresponds to a pencil $\calP_C$ of plane cubics with constant $J$-invariant 
(and vice-versa), with
\[
J=\frac{2^8(1-\lambda+\lambda^2)^3}{\lambda^2(1-\lambda^2)}. 
\] 
In particular, $J=0$ corresponds to $\lambda$ 
such that $1-\lambda+\lambda^2=0$.

We now explain how an explicit equation for $C$ 
can be obtained when $\calP_C$ is such that $J= 0$. 

\subsubsection{Chisini's theorem} 
The following theorem is due to Chisini in  \cite[\S 5]{chisini}:

\begin{theorem}\label{prop:chisiniEq}
Fix coordinates $(x:y:z)$ in $\mathbb{P}^2$ 
and let $D$ be any plane cubic that does not contain the point $p=(0:0:1)$. 
If, in these coordinates, $D$ has equation $\phi_3=0$ 
and $C$ is the plane quartic given by the equation
\begin{equation}
f_4\coloneqq \frac{\partial^2\phi_3}{\partial z^2}\phi_3-\frac 12 \left(\frac{\partial \phi_3}{\partial z}\right)^2=0,
\label{quartic}
\end{equation}
then the pencil of lines through the point $p$ 
cuts out on $C$ equianharmonic quadruples of points, 
i.e.,  their cross-ratio $\lambda$ is constant 
and is such that $1-\lambda+\lambda^2=0$; 
hence $C$ is equianharmonic as in Definition \ref{lambda-special}.
\label{main}
\end{theorem}

\begin{proof}
First, we observe that $D$ is the first polar of $C$ with respect to the point $p$. 
That is, $\frac{\partial f_4}{\partial z}=\phi_3$. 
In particular, by definition, if $L_{C,q}$ denotes the tangent line to $C$ 
at a point $q\in C$ we have that
\begin{equation}
C\cap D = \{\text{$q\in C$\,;\, $p\in L_{C,q}$}\}.
\label{int}
\end{equation}

Now, consider the conic $\psi_2=0$, 
which is the first polar of $D$ with respect to the point $p$. 
In other words, let $\psi_2$ be the degree two polynomial 
$\frac{\partial \phi_3}{\partial z}=\frac{\partial^2 f_4}{\partial z^2}$. 
Then the six points $p_1, \ldots,p_6$ lying in both the conic $\psi_2=0$ 
and the cubic $D$ are such that 
\[
\psi_2(p_i)=\frac{\partial \phi_3}{\partial z}(p_i)=\frac{\partial^2 f_4}{\partial z^2}(p_i)=0
\]
 and 
\[ 
 \phi_3(p_i)=\frac{\partial f_4}{\partial z}(p_i )=0.
 \]
In particular, $f_4(p_i)=0$ and by (\ref{int}) it follows that 
$ C\cap D=\{p_1, \ldots,p_6\}$  
and that the points $p_1, \ldots, p_6$ are flexes for $C$ 
and the flex tangent lines pass through $p$. 
Corresponding to these lines one has $J=0$. 
Moreover, since there are no other lines that intersect $C$ in non-reduced quadruples of points, the value $J=1$ is never attained when we move the line 
in the pencil with centre $p$. 
Hence $J$ has to be constant equal to $0$ when we move the line in the pencil with centre $p$.
\end{proof}

Conversely, we also have the following.

\begin{theorem}[{\cite{chisini}, {\S 5}}]\label{them:chisiniEqIFF}
If $C$ is an equianharmonic plane quartic, 
then there exists coordinates $(x:y:z)$ in $\mathbb{P}^2$ 
such that $p=(0:0:1)$ and  the equation of $C$ is  as in (\ref{quartic}). 
\end{theorem}

\begin{proof}
By definition, there exists a pencil of lines through a point $p$, 
say $\Lambda_{p}$,  
that cuts $C$ at equianharmonic quadruples of points. 
Choose coordinates $(x:y:z)$ in $\bbP^2$ such that $p=(0:0:1)$ 
and consider the first polar of $C$ with respect to $p$, 
which in these coordinates is given by $\phi_3=0$. 
Then, by definition, 
if in these coordinates $C$ is given by $f_4=0$, 
we have that $\frac{\partial f_4}{\partial z}=\phi_3$. 

Next, consider the quartic $C'$ given (in these coordinates) by
\[
f'_4\coloneqq \frac{\partial^2\phi_3}{\partial z}\phi_3-\frac 12 \left(\frac{\partial \phi_3}{\partial z}\right)^2=0.
\]
By Proposition \ref{main}, 
the pencil $\Lambda_{p}$ also cuts out on $C'$ 
equianharmonic quadruples of points; 
and, moreover, $\frac{\partial f'_4}{\partial z}=\phi_3$. 
Therefore, Lemma \ref{bin} below implies that 
given any line $L$ of $\Lambda_{p}$, 
we have that $L\cap C=L\cap C'$. 
So, it must be the case that the two plane quartics $C$ and $C'$ coincide. 
\end{proof}

\begin{lemma}
Let $\varphi_4(u,v)$ and $\psi_4(u,v)$ be two non--zero binary quartics 
which are both equianharmonic and have the same binary cubic 
as their first polar with respect to the point $(0,1)$. 
Then $\varphi_4=\alpha \cdot \psi_4$, with $\alpha\in \mathbb C$.
\label{bin}
\end{lemma}

\begin{proof}
We write
\begin{align*}
\varphi_4(u,v)&=a_0u^4+a_1u^3v+a_2u^2v^2+a_3uv^3+a_4v^4, \quad \text{and}\\
\psi_4(u,v)&=b_0u^4+b_1u^3v+b_2u^2v^2+b_3uv^3+b_4v^4.
\end{align*}
By assumption, $\frac{\partial \varphi_4}{\partial v}=\frac{\partial \psi_4}{\partial v}$ so there exists some $\alpha \neq 0$ such that $a_i=\alpha\cdot b_i$ for $1\leq i\leq 4$. Now, because both $\varphi_4(u,v)$ and $\psi_4(u,v)$ are equianharmonic, Lemma \ref{invI} below implies that
\[
12a_0a_4-3a_1a_3+a_2^2 = 12b_0b_4-3b_1b_3+b_2^2 =0.
\]
Thus, by  replacing $a_i=\alpha\cdot b_i$ for $1\leq i\leq 4$, we obtain that $a_0=\alpha b_0$.
\end{proof}

\begin{lemma}
If a binary quartic $a_0u^4+a_1u^3v+a_2u^2v^2+a_3uv^3+a_4v^4$ is equianharmonic, then
\[
I:=12a_0a_4-3a_1a_3+a_2^2=0.
\]
\label{invI}
\end{lemma}

\begin{proof}
To any (nonzero) binary quartic as above, we can associate four points on $\mathbb{P}^1$ (the roots). If the points are all distinct, up to a change of coordinates, we may assume that they are $\{0, 1, \infty, \lambda\}$. Then $\lambda$ is (up to ordering) the cross-ratio of the four points, and one can show  (see, e.g., \cite{invariants}, \cite{dimca}, \cite[Remark 5.2]{chilean} and  \cite[Vol. I, p. 29]{EC}) that
\[
1-\lambda+\lambda^2=0 \iff I=0.
\]
\end{proof}

We have then:

\begin{theorem}\label{thm:chisini+}
Let $S$ be a RESS of special type $II$ and let  $(C,p)$ be associated to it. Then:
\begin{enumerate}[(i)]
    \item $P \notin C$;
    \item fix coordinates $(x:y:z)$ in $\mathbb{P}^2$ such that $p=(0:0:1)$. 
Let  $\phi_3=0$ be the equation of a general plane cubic $D$ that does not contain the point  $p$.
    Then the equation of $(C,p)$ has a normal form 
\begin{equation}
f_4\coloneqq \frac{\partial^2\phi_3}{\partial z^2}\phi_3-\frac 12\left(\frac{\partial \phi_3}{\partial z}\right)^2=0.
\label{quarticbis}
\end{equation}
\end{enumerate}
\end{theorem}

\begin{proof} (i) follows from Proposition \ref{lem:PinC}, 
while Theorem \ref{them:chisiniEqIFF} proves (ii).
\end{proof}

\begin{corollary}\label{cor:norm}
We maintain the notation of Theorem \ref{thm:chisini+}.   
Let $S$ be a RESS of special type $II$ and let $(C,p)$ be associated to it. 
Then  the equation of $(C,p)$ has a normal form 
 \begin{equation*}
\frac{\partial^2\phi_3}{\partial z^2}\phi_3-\frac 12 \left(\frac{\partial \phi_3}{\partial z}\right)^2=0, \text{ where }
\phi_3=x^3+y^3+z^3-3\gamma xyz=0
\label{quarticcubiss}
\end{equation*}
and $\gamma$ is general, with $p=(0:0:1)$.
\end{corollary}

\begin{proof}
 Any smooth plane cubic is projectively isomorphic to a member of the Hesse pencil. Thus, with notations as in Proposition \ref{main}, it is relevant to our purposes to consider the cubics of the form $\phi_3=x^3+y^3+z^3-3\gamma xyz=0$. 
If $\gamma$ is general, then the equianharmonic quartic $f_4=0$ is smooth, and there are precisely six concurrent flex lines at $p=(0:0:1)$.  
\end{proof}

\begin{example}[$\gamma=4$]\label{ex:hesseCuartic}
If one takes $\phi_3=x^3+y^3+z^3-12xyz$, which defines a plane cubic lying in the Hesse pencil, then 
\[
f_4=6 x^3 z-72 x^2 y^2-36 x y z^2+6 y^3 z+\frac{3 z^4}{2}.
\]

The six flex points have the form $(-16\zeta^5+32\zeta^2:\zeta:1)$, where $2\zeta$ is one of the following six algebraic numbers: 
\begin{multicols}{2}
\begin{itemize}
    \item $-(-8-3\sqrt{7})^{1/3}$,
    \item $-(-8+3\sqrt{7})^{1/3}$, 
    \item $(8-3\sqrt{7})^{1/3}$, 
    \item $(8+3\sqrt{7})^{1/3}$, 
    \item $\omega(8-3\sqrt{7})^{1/3}$ or 
    \item $\omega(8+3\sqrt{7})^{1/3}$;
\end{itemize} 
\end{multicols}
and $\omega=e^{2\pi i/3}$.
\end{example}

\begin{remark}   For special values of $\gamma$, namely $\gamma\in \{0,1,\omega,\omega^2\}$ with $\omega=e^{2\pi i/3}$, the curve defined by $f_4=0$ is singular.
  For example, if  $\phi_3$ is the Fermat cubic $x^3+y^3+z^3=0$, then $C$ is reducible. $C$ is the union of the Fermat cubic and the line $z=0$  which intersect in $3$ distinct points. The corresponding rational elliptic surface will have three singular fibres of type  $IV$. It is a non-special example with $J=0$.
\end{remark}

\subsection{A different approach} In this section, we give a way, different from Chisini's one, to determine the pencils of cubics such that all their smooth members are equianharmonic.

Let us start by considering a cubic curve in the plane defined by the equation
\[
Ax^3 + By^3 + Cz^3 + Px^2y + Qy^2z + Rz^2x + Txy^2 +Uyz^2 + Vzx^2 + Mxyz = 0.
\]
Following \cite{ART}, we define
\begin{align*}
a_1 &= M \\
a_2 &= -(PU + QV + RT) \\
a_3 &= 9ABC - (AQU + BRV + CPT ) - (TUV + PQR) \\
a_4 &= (ARQ^2 + BPR^2 + CQP^2 + ATU^2 + BUV^2 + CVT^2) \\
    &\text{ }\;\; +(PQUV + QRVT + RPTU) - 3(ABRU + BCPV + CAQT) 
\end{align*}
and then
\[
b_2 = a_1^2+4a_2 \;\;\;\text{ and }\;\;\; b_4 = 2a_4+a_1a_3.
\]
Then
\[
c_4 = b_2^2-24b_4
\]
is a multiple of the $A$ coefficient of a Weierstrass equation for the cubic.
Hence, $c_4=0$ is the condition for the cubic to have $J=0$.

One can  apply this to a pencil, generated by two distinct cubics with coefficients
$\{A,B,\\
\ldots,M\}$ and $\{A_1,B_1,\ldots,M_1\}$; one applies the above formulas to the set of coefficients (depending on the pencil parameter $t$) which is $\{A+tA_1,\ldots,M+tM_1\}$.
This gives $c_4$ as a polynomial (of degree four in $t$), and if we want $c_4$ identically zero, we would set the five coefficients of this polynomial equal to zero.

\begin{example}\label{ex:1}
It is an exercise to type this into a symbolic algebra package (such as SageMath)
and make computations.  We did this, in general, and applied it to the pencil generated by $x^3+y^3+z^3=0$
and $Px^2y + Qy^2z + Rz^2x=0$; the answer in this case is that
\[
c_4 = -48(P^2Q + Q^2R + PR^2) t^3.
\]

\end{example}

\begin{example}\label{ex:2}
We applied the general formula
to a pencil generated by the cuspidal cubic $x^3-y^2z$
and a general one with coefficients $\{A,\ldots,M\}$ as above.
The resulting quartic polynomial $c_4(t)$ has no constant term,
and the linear term is $-48Rt$.
Hence, we must have $R=0$ to have a constant $J=0$ pencil.
Substituting that into the other coefficients of $c_4$ we arrive at
\begin{align*}
c_4 &= -(216ABCM - M^4 + 48(CP^2Q+ ATU^2+ BUV^2) \\
 & - 24(CMPT + AMQU + MTUV) + 8(M^2PU + M^2QV) - 16P^2U^2  \\
& - 144(BCPV + ACQT)  + 48CT^2V + 16(PQUV  - Q^2V^2) )t^4 \\
 & - 8(27BCM - 6CP^2 + 6TU^2+ 18ACT - 18CQT + 3AMU - 3MQU\\  
& - M^2V - 2PUV + 4QV^2)t^3 \\
 & - 8(18CT + 3MU - 2V^2)t^2.
\end{align*}
By adding an appropriate multiple of the cuspidal equation to the other generator,
we can also assume that $A=0$; if we do that it simplifies to 
\begin{align*}
c_4 &= (M^4 - 48CP^2Q + 24CMPT - 8M^2PU + 16P^2U^2 + 144BCPV - 8M^2QV \\
&- 48CT^2V - 16PQUV + 24MTUV + 16Q^2V^2 - 48BUV^2)t^4 \\
&  - 8(27BCM - 6CP^2 - 18CQT - 3MQU + 6TU^2 - M^2V - 2PUV + 4QV^2)t^3 \\
&  - 8(18CT + 3MU - 2V^2)t^2.
\end{align*}

\end{example}

\begin{remark} The locus of the six cuspidal cubics in an pencil of elliptic plane curve is   the discriminant of a  (non minimal) Weiestrass elliptic fibration $W \to \mathbb{P}^2$, where $W$ is birational to a Calabi-Yau threefold with $\operatorname{rk}MW(W/\mathbb{P}^2)=10$. This gives  the highest known rank of a Mordell-Weil group of a Calabi-Yau threefold \cite{GrassiWeigand2022}.
The search of such effective bounds is of interest to  the physics of phenomenology and string theory.  N. E. Elkies   constructed a family of  elliptic fibrations with numerically trivial canonical bundle over $\mathbb{P}^2$ \cite{Elkies2018}. Elkies does not claim however that the  threefolds are birationally Calabi-Yau. To prove that the model  $W\to \mathbb{P}^2$ is  in Elkies' family, one would have to prove that  the equation of the product of the six cuspidal curves is the restriction of an invariant of the Burkhardt group  $\operatorname{Sp}(4, \mathbb F_3) \simeq G_{25920}$.
\end{remark}

\section{The mixed case {$(4,2)$}}\label{sec:mix1} 

\subsection{The Weierstrass equation and moduli}
In this case, the four cuspidal fibres occur at points of common roots to the Weierstrass coefficients $A$ and $B$.
Therefore, since $A$ is a quartic, it must divide $B$, 
and we may write $B=AQ$ for a quadratic polynomial $Q$.
Then we have the discriminant
\[
D = 4A^3 + 27B^2 = A^2(4A+27Q^2)
\]
and in order that we have the condition that the remaining singular fibres are two $I_2$'s,
we need $D$ to be a perfect square; this requires us to have a quadratic polynomial $P$ such that 
\[
P^2 = 4A + 27Q^2.
\]
Conversely, if we have any two quadratics $P$ and $Q$, we set
\[
A = (P^2-27Q^2)/4 \;\;\;\text{ and }\;\;\; B = AQ
\]
and for $P$ and $Q$ general, we will have a rational elliptic surface
with the desired configuration of singular fibres.

The six parameters of $P$ and $Q$ give us two moduli,
factoring out the automorphisms of the base curve and the scaling in the Weierstrass equation.

\subsection{The double cover models} \label{ssec:dcm} 
Looking at the double cover model of $\bbF_2$, we have here a double cover of $\bbF_2$ 
branched along the negative section $B$ 
and a trisection $T$ in the linear system $|3B+6F|$ disjoint from $B$. 
This trisection has two double points and has four flexes. 

Let us map  $\bbF_2$ to $\bbP^3$ via the linear system $|B+2F|$. 
The image is a quadric cone $Q$ 
and the double cover is branched along the vertex of $Q$ and the image of $T$, 
that we still denote by $T$ for convenience. 
The curve $T$ of degree $6$ has two nodes and $4$ lines of the ruling of the cone as flex tangents. The geometric genus of $T$ is 2. 

Let $x$ be one of the nodes of $T$, 
and project $Q$ down to $\bbP^2$ birationally from $x$. 
The curve $T$ is mapped to  a plane quartic $C$ with a node $q$ 
(the image of the other node of $T$). 
The image of the vertex of the cone $Q$ is a point $p$ of the plane 
and $C$ passes through  $p$. 
Then we have the double cover of $\bbP^2$ ramified model of the RESS 
with the associated pair $(C,p)$. 
It turns out that there are four distinct  lines passing through $p$ 
that are flex tangents to $C$ off $p$. 

It would be possible to find a normal form for such a pair $(C,p)$, but we do not dwell on this here. 

\section{The mixed case $(3,3)$}\label {sec:mix2}
\subsection{The Weierstrass equation and moduli}
Here we may choose coordinates on $\bbP^1$
so that the three $II$ fibres are at $0,1,\infty$.
In that case, in the Weierstrass equation we will have
\[
A = t(t-1)(t-\lambda) \;\;\text{ and }\;\; B = t(t-1)P(t)
\]
where we have used up the scaling parameter
and the polynomial $P$ has degree three, 
and does not have $0,1,\infty$ as roots.
The discriminant $D(t)$ then factors as
\[
D = t^2(t-1)^2[4t(t-1)(t-\lambda)^3 + 27P(t)^2]
\]
and in order to have the three $I_2$ fibres, this must be a perfect square.
Hence, there is a cubic polynomial $Q(t)$ such that
\[
4t(t-1)(t-\lambda)^3 + 27P(t)^2 = Q(t)^2
\]
or
\[
4t(t-1)(t-\lambda)^3 = Q^2-27P^2 = (Q+rP)(Q-rP)
\]
where $r = \sqrt{27}$.

Now we note that $(t-\lambda)$ cannot be a factor of both $Q+rP$ and $Q-rP$,
since if so, it would be a factor of their sum and difference,
and hence a common factor of $P$ and $Q$.
This contradicts the fact that $P(\lambda) \neq 0$.

Therefore, $(t-\lambda)^3$ must be equal (up to a constant factor) 
to either $Q+rP$ or $Q-rP$, 
and we may assume it is $Q+rP$ (for one of the two roots $r$ of $27$).

This means that we have
\[
Q + r P = \alpha(t-\lambda)^3 \;\;\text{ and }\;\; Q-rP = \beta t(t-1)
\]
for some constants $\alpha, \beta$, and we must have $\alpha\beta = 4$.
Hence we have
\[
2rP = \alpha(t-\lambda)^3 - \beta t(t-1)
\]
or
\[
P(t) = \frac{\alpha(t-\lambda)^3}{2r} - 2 \frac{t(t-1)}{r\alpha}.
\]
This determines the Weierstrass equation using the two parameters $\alpha,\lambda$;
we see that there are two moduli for the family.

\subsection{An alternative description of the double cover}\label{ssec:ddcc}
Let $T$ be the trisection of $\bbF_2$ determining the double cover representation.
In this case, $T$ has three nodes and is flexed to three fibres.
There is a $(+2)$-section $S$ passing through the three nodes.

Perform three elementary transformations of the $\bbF_2$,
at the three nodes.
The section $S$ becomes a $(-1)$-section; the resulting ruled surface is an $\bbF_1$.
If we blow down $S$, we arrive at the plane, and 
the trisection $T$ becomes a smooth cubic curve.
This curve has the three flexes on it,
and the three flexed tangent lines are concurrent at the point we blow up
to obtain the $\bbF_1$.

It is a standard fact that the only cubics with three concurrent flexed tangent lines
are the ones with $J=0$ (see, e.g.,  \cite[Vol. II, p. 234] {EC}).

Now we may reverse the above and construct the double cover representation
by first choosing the $J=0$ cubic curve in the plane.
There are no moduli here.
Then we choose three concurrent flexed tangents, and blow up the point of concurrency
(which is not on the cubic).
This produces $\bbF_1$; the cubic is disjoint from the $(-1)$-curve.

Now we perform three elementary transformations at three collinear points of the proper transform of the cubic.
There are two moduli for this.
This produces the trisection $T$, and gives the resulting elliptic surface.

\subsection{The double plane model} 
Take an equianharmonic cubic $\Gamma$ 
that has three flexes concurring at a point $p$. 
Take a line $L$ that does not pass through any of the above flexes. 
Set $C=\Gamma+L$. 
Then the pair $(C,p)$ is associated to a special RESS of type $(3,3)$ 
and, taking into account the discussion in Section \ref {ssec:ddcc}, we see that we obtain the general such RESS in this way. Again, we see that the number of moduli is two (the choice of the line $L$). 
This also leads to a normal form for the pair $(C,p)$. 
For example we may take for $\Gamma$ the Fermat cubic $x^3+y^3+z^3=0$, 
for $p$ the point $(1:1:1)$ 
(the intersection of the three flex tangents $x-y=0, x-z=0, y-z=0$), 
and for $L$ a general line $ax+by+cz=0$. 

There is another way to arrive at a double-plane model, as follows. 
As in Section \ref {ssec:ddcc}, 
we have the trisection $T$ of $\bbF_2$ 
determining the double cover representation.
Then $T$ has three nodes and is flexed to three fibres.

As in \ref{ssec:dcm}, realize  $\bbF_2$ in $\bbP^3$ as a quadric cone $Q$ 
and consider the image of $T$, which we still denote by $T$, by abuse of notation. The curve $T$ of degree $6$ has three nodes 
and $3$ lines of the cone as flex tangents. 
The geometric genus of $T$ is $1$. 

Let $x$ be one of the nodes of $T$, 
and project $Q$ down to $\bbP^2$ birationally from $x$. 
The curve $T$ is mapped to a plane quartic $C$ with two nodes 
(the images of the other two nodes of $T$). 
The image of the vertex of the cone $Q$ is a point $p$ of the plane 
and $C$ passes through $p$. 
Then we have the double cover of $\bbP^2$ ramified model of the RESS 
with the associated pair $(C,p)$. 
It turns out that there are three distinct lines passing through $p$ 
that are flex tangents to $C$ off $p$. 

Again, we do not dwell here on finding a normal form for the pair $(C,p)$. 

\section{The mixed case $(2,4)$}\label {sec:mix3}

\subsection{A model for the cubic pencil}
Again, we let $S_0$ denote the zero section of the elliptic surface $\calX$.
In this cas,e we can denote the components of the four $I_2$ fibres 
by $C_i, D_i$, $1 \leq i \leq 4$, where $D_i$ meets $S_0$; all are $(-2)$-curves.
The $C_i$'s generate the rank four sublattice $R \subset U^\perp$,
and recall that $U^\perp/R \cong \MW$.

%
%
%
%

To indicate how the corresponding cubic pencils may look,
we consider the elliptic surface $\calX$ with zero section $S_0$, 
blow it down, and arrive at the Del Pezzo surface of degree one.
On this surface, as noted above, the collection of ($-1)$-curves
form the $240$ roots of an $E_8$ lattice 
via the map $\sigma$ (see Section \ref {NSMW});
moreover, disjoint $(-1)$-curves correspond to roots which meet in $-1$.
The Weyl group acts transitively on subsets of $8$ roots which meet in $-1$.
It also acts transitively on subsets of $3$ roots which are orthogonal, 
as noted in Section \ref {NSMW}.

On $\calX$ we have four $C_i$'s, 
which correspond to four double points of the trisection branch curve on $\bbF_2$.
Given any three of them, we may consider the $(+2)$-section of $\bbF_2$
passing through them.  
This section lifts to two disjoint sections on $\calX$, each meeting the three given $C_i$'s once.

These three $C_i$'s are classes in $U^\perp \cong E_8$
which are orthogonal roots.  
Hence, by the transitivity of the Weyl group action, we may choose them to be the following three elements of $E_8$
(using the negative definite $\bbR^8$ description):
\[
\begin{array}{ccccccccc}
C_1: & 1 & 1 & 0 & 0 & 0 & 0 & 0 & 0 \\
C_2: & 1 & -1 & 0 & 0 & 0 & 0 & 0 & 0 \\
C_3: & 0 & 0 & 1 & 1 & 0 & 0 & 0 & 0 
\end{array}
\]
We now look for the two disjoint sections as constructed above
which meet these three $C_i$'s; these correspond to roots in $E_8$
which intersect each of these three $C_i$'s once.
In fact, they are uniquely determined geometrically;
either of these sections map to the unique $(+2)$-section of $\bbF_2$.
It is easy to see that the only two roots that can work are:
\[
\begin{array}{ccccccccc}
u_a: & 1 & 0 & 1 & 0 & 0 & 0 & 0 & 0 \\
u_b: & 1 & 0 & 0 & 1 & 0 & 0 & 0 & 0 \\
\end{array}
\]
These two roots do not meet the $C_4$ 
(and also $C_4$ is disjoint from the three $C_i$'s), so it must be the case 
(up to permutation of the final four coordinates) that $C_4$ is
\[
\begin{array}{ccccccccc}
C_4: & 0 & 0 & 0 & 0 & 1 & 1 & 0 & 0 \\
\end{array}
\]
We are now in a position to find four roots $u_i$ ($1\leq i \leq 4$)
such that $u_i \cdot C_j = \delta_{ij}$ and  $u_i \cdot u_j = 1$ for $i \neq j$.
The following work:
\[
\begin{array}{ccccccccc}
u_1: & 1/2 & 1/2 & 1/2 & -1/2 & 1/2 & -1/2 & 1/2 & 1/2 \\
u_2: & 1/2 & -1/2 & 1/2 & -1/2 & 1/2 & -1/2 & 1/2 & -1/2 \\
u_3: & 0 & 0 & 1 & 0 & 0 & 0 & 1 & 0 \\
u_4: & 0 & 0 & 0 & 0 & 1 & 0 & 1 & 0 \\
\end{array}
\]
If we let $S_i = \sigma(u_i)$ for each $i$,
then the $S_i$'s are disjoint sections of $\calX$,
each $S_i$ meeting only $C_i$.
Therefore, we can blow $\calX$ down to the plane
by blowing down first $S_0$,
then the $S_i,C_i$ pair for each $i$ ($1\leq i \leq 4$).
These nine blowdowns arrive at the plane, and the two $II$ fibres go to cuspidal cubics, of course.
These two cuspidal curves generate the pencil,
and the base points are the one simple base point 
(to which $S_0$ is blown down)
and four infinitely near double base points 
(to which the $S_i,C_i$ pairs are blown down).
The four $I_2$ fibres go to nodal cubics,
each with a node at the double base point; the nodal cubic is the image of the $D_i$ fibre component, of course.

This proves the following.

\begin{theorem}
Given a RESS $(\calX,S_0)$ with two $II$ fibres and four $I_2$ fibres,
$\calX$ may be blown down to a pencil of cubics
generated by two cuspidal cubics which are simply tangent at four points, and meet once more transversally.
Any two cubic pencils which resolve to $\calX$ are Cremona equivalent,
and any cubic pencil that resolves to $\calX$ is Cremona equivalent
to a four-tangent cuspidally generated pencil.
\end{theorem}

\subsection{The number of moduli}
There are two moduli for these cubic pencils.
To see this, we first fix one of the cuspidal cubics.
There is a one-parameter family of automorphisms of the plane fixing that cubic.
To choose the second cubic, 
we note that the family of cuspidal cubics is $7$-dimensional;
and hence those that are tangent to the chosen cuspidal cubic four times are three-dimensional.

Then the automorphisms reduce the number of moduli to two.

\subsection{An alternative description}\label{ssec:alt}
Consider the trisection $T$ of $\bbF_2$ formed by the double cover representation.
In this case, $T$ has four nodes and is flexed to two fibres,
giving the four $I_2$ and two $II$ fibres.

Make three elementary transformations based at three of the nodes.
In this case the ruled surface $\bbF_2$ is transformed into an $\bbF_1$;
the negative curve on the $\bbF_1$ comes from the $(+2)$-section of the $\bbF_2$
which passes through the three nodes.
At this point the trisection $T$ is disjoint from this $(-1)$-curve and has only one node;
when we blow this $(-1)$-curve down, $T$ becomes a nodal cubic
and the base point of the pencil of lines that generate the ruling on the $\bbF_1$
does not lie on $T$.
Two of the lines in this pencil are flex tangent lines to $T$.

We reverse this construction to recover the original double-cover representation. This means choosing a nodal cubic (there are no moduli here),
and two flex lines to it (there are also no moduli for this).
We also choose two smooth points on $T$, say $p,q$; 
let $r$ be the third point of intersection of $T$ with the line joining $p$ and $q$.
Then we blow up the intersection point of the flex lines to obtain $\bbF_1$;
then we perform three elementary transformations at $p$, $q$, and $r$.
The line joining them becomes the $(-2)$-section of the resulting $\bbF_2$,
and the proper transform of $T$ is the trisection that gives $\calX$ in the double-cover representation.

This shows that there are two moduli: 
the choice of the two smooth points $p$ and $q$.

\subsection{The double plane representation} 
A different way to interpret the construction in Section \ref{ssec:alt} is as follows. Take $\Gamma$ a nodal cubic. 
There are three (collinear) flexes of $\Gamma$. 
Take two of them and the intersection $p$ of the flex tangents at these two points.  Then take a line $L$ that does not contain the two chosen flexes and the node. 
Set $C=\Gamma+L$. 
The pair $(C,p)$ is associated with a special RESS of type $(2,4)$. The discussion in Section \ref {ssec:alt} shows that, in this way, we obtain the general such RESS. 
This leads us to a normal form for the pair $(C,p)$. 
For example, we can take for $\Gamma$ the plane cubic 
with equation $xyz+x^3+y^3=0$ that has a node at $(0:0:1)$. 
The three collinear flexes of $\Gamma$ 
are the points $(1:-1:0)$, $(1\pm \sqrt 3i:2:0)$. 
The intersection point of the tangents at the last two is $p=(1:1:-3)$. 
Then we can take as $L$ a line $ax+by+cz=0$
that does not contain the points $(0:0:1)$ and $(1\pm \sqrt 3i:2:0)$.

A different way of constructing a double plane representation is the following. Consider again the trisection $T$ of $\bbF_2$ given by the double cover.
It has four nodes and is flexed to two fibres. 
Now realize $\bbF_2$ as a quadric cone in $\bbP^3$ 
and project down to a plane from a double point of $T$. 
Then $T$ maps to a degree $4$ rational curve with three nodes, 
and two flex tangents passing through a point $p\in C$, 
the image of the vertex of the quadric cone via the projection.
The RESS is now associated with the pair $(C,p)$. 

From this last model, we can see again that the number of moduli is two. 
Indeed, we can choose coordinates in the plane 
so that the three nodes of $C$ and the point $p$ 
are the three coordinate points and the $(1:1:1)$ point. 
The linear system $\calC$ of quartics that have this behaviour has dimension $4$. 
Now take two general lines $r_1,r_2$ through the point $p$ 
and let us require that these are flex tangent lines to a curve of $\calC$. 
It turns out that there are only finitely many such curves in $\calC$, 
so that the number of moduli is two, 
i.e., the choice of the two lines $r_1,r_2$. 
Indeed, the linear system $\calC$ maps $r_1\cup r_2$ 
to a pair of disjoint rational normal cubics $R_1, R_2$ in $\bbP^4$, 
lying in distinct hyperplanes $\Pi_1$ and $\Pi_2$, 
that intersect in a plane $P$. 
The curves of $\calC$ correspond to hyperplanes of $\bbP^4$. 
For example, the hyperplanes $\Pi_1$ and $\Pi_2$ 
correspond to the curves of $\calC$ containing $r_1$ and $r_2$, 
that consist of $r_1$ or $r_2$ plus the triangle formed 
by the lines pairwise joining the coordinate points. 
We must find the hyperplanes that intersect both cubics $R_1, R_2$ 
in a divisor consisting of one point with multiplicity $3$. 
We claim that there are only finitely many such hyperplanes. 
To see this, consider the map $\varphi_i: R_i\longrightarrow P^\vee$ 
that sends a point $x\in R_i$ 
to the intersection line of the osculating plane to $R_i$ at $x$ in $\Pi_i$ with $P$, 
for $1\leq i\leq 2$. 
The images of these two maps are two curves $D_1$ and $D_2$ in $P^\vee$. It is not difficult to see that, in fact, these two curves have degree three, though this will not be important for us. 
Then $D_1$ and $D_2$ intersect in finitely many (in fact nine) points, and correspondingly there are finitely many (in fact $9$) lines in $P$, such that two osculating planes to $R_1$ and $R_2$ intersect along one of these lines. 
The span of a pair of such planes is one of the required hyperplanes.

\end{document}